\documentclass{amsart}
\usepackage[foot]{amsaddr}
\usepackage{amssymb,amsmath,amsfonts,mathptmx,cite,mathrsfs,graphicx,gastex,rotating,url,mathrsfs}
\usepackage{eucal}

\theoremstyle{plain}
\newtheorem{theorem}{Theorem}
\newtheorem{proposition}{Proposition}
\newtheorem{lemma}{Lemma}
\newtheorem{corollary}{Corollary}

\theoremstyle{remark}
\newtheorem{remark}{Remark}

\def\softd{{\leavevmode\setbox1=\hbox{d}%
    \hbox to 1.05\wd1{d\kern-0.4ex{\char039}\hss}}}

\DeclareSymbolFont{rsfscript}{OMS}{rsfs}{m}{n}
\DeclareSymbolFontAlphabet{\mathrsfs}{rsfscript}

\newcommand{\mA}{\mathcal{A}}
\newcommand{\mB}{\mathcal{B}}
\newcommand{\mC}{\mathcal{C}}
\newcommand{\mE}{\mathcal{E}}

\newcommand{\mN}{\mathcal{N}}

\newcommand{\mR}{\mathcal{R}}
\newcommand{\mS}{\mathcal{S}}
\newcommand{\mT}{\mathcal{T}}

\def\fb{finitely based}
\def\ib{identity basis}
\def\nfb{non\-finitely based}
\newcommand{\ais}{ai-semi\-ring}

\DeclareMathOperator{\alf}{alph}

\DeclareMathOperator{\End}{End}
\DeclareMathOperator{\Id}{Eq}
\DeclareMathOperator{\var}{var}

\usepackage{tikz}
\usetikzlibrary{arrows}
\usetikzlibrary{arrows,shapes}
\usetikzlibrary{tqft}

\makeatletter

\renewcommand*\subjclass[2][2010]{\def\@subjclass{#2}\@ifundefined{subjclassname@#1}{\ClassWarning{\@classname}{Unknown edition (#1) of Mathematics Subject Classification; using '2010'.}}{\@xp\let\@xp\subjclassname\csname subjclassname@#1\endcsname}}

\renewcommand{\subjclassname}{\textup{2010} Mathematics Subject Classification}

\makeatother

\allowdisplaybreaks
\begin{document}

\title[The finite basis problem for endomorphism semirings]{The finite basis problem for endomorphism semirings\\ of finite chains}
\thanks{Supported by the Russian Science Foundation (grant No.\ 22-21-00650) and the Ministry of Science and Higher Education of the Russian Federation (project FEUZ-2023-0022).}

\author[S. V. Gusev]{Sergey V. Gusev}
\address{{\normalfont Institute of Natural Sciences and Mathematics, Ural Federal University 620000 Ekaterinburg, Russia}}
\email{sergey.gusb@gmail.com}
\email{m.v.volkov@urfu.ru}

\author[M. V. Volkov]{Mikhail V. Volkov}

\begin{abstract}
For every semilattice $\mS=(S,+)$, the set $\End(\mS)$ of its endomorphisms forms a semiring under point-wise addition and composition. We prove that the semiring of all endomorphisms of the 3-element chain has no finite identity basis. This, combined with earlier results by Dolinka (The finite basis problem for endomorphism semirings of finite semilattices with zero, Algebra Universalis 61, 441--448 (2009)), gives a complete solution to the finite basis problem for semirings of the form $\End(\mS)$ where $\mS$ is a finite chain.
\end{abstract}



\maketitle

\section{Background, motivation, and overview}
\label{sec:introduction}

A semigroup $\mS=(S,+)$ is called a \emph{semilattice} if it is commutative and idempotent, that is, the equalities $s+t=t+s$ and $s+s=s$ hold for all $s,t\in S$. The set $\End(\mS)$ of all endomorphisms of $\mS$ is closed under  composition: for all endomorphisms $\alpha,\beta\colon S\to S$, their composition $\alpha\beta$, whose action at each $s\in S$ is defined by $s(\alpha\beta):=(s\alpha)\beta$, is easily seen to be an endomorphism as well. The point-wise sum $\alpha+\beta$ defined by $s(\alpha+\beta):=s\alpha+s\beta$ also is an endomorphism of $\mS$. Indeed, for all $s,t\in S$ and $\alpha,\beta\in\End(\mS)$,
\begin{align*}
  (s+t)(\alpha+\beta)&=(s+t)\alpha+(s+t)\beta&&\text{by the definition of $\alpha+\beta$}\\
                     &=(s\alpha+t\alpha)+(s\beta+t\beta)&&\text{since $\alpha$ and $\beta$ are endomorphisms}\\
                     &=(s\alpha+s\beta)+(t\alpha+t\beta)&&\text{since $(S,+)$ is a commutative semigroup}\\
                     &=s(\alpha+\beta)+t(\alpha+\beta)&&\text{by the definition of $\alpha+\beta$}.
\end{align*}
Clearly, $\End(\mS)$ is a semilattice with respect to this addition and a semigroup under composition. It is straightforward that composition distributes over point-wise addition on the left and on the right, that is,
\[
\alpha(\beta+\gamma)=\alpha\beta+\alpha\gamma\ \text{ and }\ (\alpha+\beta)\gamma=\alpha\gamma+\beta\gamma\ \text{ for all }\ \alpha,\beta,\gamma\in \End(\mS).
\]

An \emph{additively idempotent semiring} (\ais, for short) is an algebraic structure $\mathcal{R}=(R,+,\cdot)$ with two binary operations, addition $+$ and multiplication $\cdot$, such that the additive reduct $(R,+)$ is a semilattice, the multiplicative reduct $(R,\cdot)$ is a semigroup, and multiplication distributes over addition on the left and on the right. Thus, $\End(\mS)$ is an \ais{} for every semilattice $\mS$. This observation was made long ago; see, e.g., \cite[Section XIV.4, Example 4]{Birkhoff}. More recently, it has been shown by Martin Ku\v{r}il and Libor Pol\'ak that endomorphism semirings of semilattices play the same universal role for \ais{}s as symmetric groups play for groups by Cayley's textbook theorem: every \ais{} embeds into the endomorphism semiring of a suitable semilattice \cite[Theorem 2.4]{KP05}\footnote{In \cite{KP05}, as well as in other papers whose motivation comes from semigroup theory (e.g., \cite{McA07}), \ais{}s appear under the name \emph{semilattice-ordered semigroups}.}. We refer to \cite{JKM09} for a study of structure properties of endomorphism semirings of semilattices, which has revealed further similarities between these semirings and symmetric groups.

Igor Dolinka~\cite{Dolinka09endo} initiated investigations of the finite axiomatizability question (aka the finite basis problem, FBP for short) for semiring identities satisfied by endomorphism semirings of finite semilattices. Recall that a set $\Sigma$ of identities valid in an \ais{} $\mathcal R$ is called an \emph{identity basis} for $\mathcal R$ if all identities holding in $\mathcal R$ can be inferred from $\Sigma$. If $\mathcal R$ admits a finite identity basis, then it is \emph{finitely based}; otherwise, $\mathcal R$ is \emph{nonfinitely based}. The FBP for a class of \ais{}s is the issue of determining which \ais{}s from this class are finitely based and which are not.

In~\cite{Dolinka09endo}, Dolinka used the concept of inherent nonfinite basability that he adapted to the \ais{} setting in~\cite{Dolinka09infb}. Semirings considered in~\cite{Dolinka09infb} had, along with addition and multiplication, the nullary operation of taking a neutral element for addition. In~\cite{Dolinka09endo}, the approach from~\cite{Dolinka09infb} was applied to the subsemiring $\End^0(\mS)$ of $\End(\mS)$ that exists whenever the semilattice $\mS$ has the least element 0: the subsemiring consists of all endomorphisms of $\mS$ that fix 0. For finite \ais{}s of the form $\End^0(\mS)$ where $\mS$ is a semilattice with~0, Theorem~7 of \cite{Dolinka09endo} solves the FBP except for the single case when $\mS$ is the 3-element chain. The question of whether or not the exceptional \ais\ is finitely based is still open.

Let $\mC_m$ stand for the $m$-element chain considered as a semilattice. In Section~\ref{sec:methods} we explain how Dolinka's results imply that each \ais{} $\End(\mC_m)$ with $m\ge 4$ is nonfinitely based. Since the $3$-element \ais{} $\End(\mC_2)$ is easily seen to be finitely based, the FBP for endomorphism semirings of finite chains reduces to the question of whether or not $\End(\mC_3)$ is finitely based. This question turns out to be rather complicated as the \ais{} $\End(\mC_3)$  does not fall into the range of previously known tools for proving the absence of a finite identity basis for a given \ais{}. Therefore, a new method had to be developed. We prepare ingredients for our approach in Sections~\ref{sec:a21} and~\ref{sec:kadourek}; in Section~\ref{subsec:plan}, we also discuss the overall proof plan. In Section~\ref{sec:semigroups-sn}, we present our main construction and establish its crucial properties; this is the technical core of the paper. Section~\ref{sec:main} presents a sufficient condition for an \ais\ to be nonfinitely based. We show that this condition applies to the semiring $\End(\mC_3)$, thus completing the solution to the FBP for the \ais{}s $\End(\mC_m)$.

The paper uses a few standard concepts of semigroup theory (such as, e.g., semigroup presentations, Green's relations, ideals, and Rees quotients) that all can be found in the early chapters of the canonical textbooks \cite{Clifford&Preston:1961} and \cite{Howie:1995}.

\section{Applying Dolinka's results to the FBP for $\End(\mC_m)$ with $m\ge 4$}
\label{sec:methods}

\subsection{Preliminaries}\label{subsec:preliminaries}
The concepts of an identity and an identity basis are intuitively clear. Nevertheless, any precise reasoning about these concepts requires a formal framework. Such a framework, provided by equational logic, is concisely presented, e.g., in \cite[Chapter~II]{BuSa81}. We briefly overview the basic vocabulary of equational logic in a form adapted to the use in this paper. Readers familiar with equational logic may skip this overview.

By an \emph{algebraic structure} we mean a tuple $\mA=(A,f_1,\dots,f_k)$ in which $A$ is a nonempty set, called the \emph{carrier} of $\mA$, and each $f_i$, $i=1,\dots,k$, is an \emph{operation} on $A$, that is, a map
\[
f_i\colon\underbrace{A\times A\times\dots\times A}_{\text{$n_i$ times}}\to A.
\]
The integer $n_i\ge 0$ is called the \emph{arity} of the operation $f_i$, and the $k$-tuple $(n_1,\dots,n_k)$ is referred to as the \emph{type} of the structure $\mA$. We allow operations of arity 0 (\emph{nullary} operations);  a nullary operation simply selects a certain element of the set $A$.

Given a $k$-tuple $\tau=(n_1,\dots,n_k)$ of nonnegative integers, $k$ symbols $\mathbf{f}_1,\dots,\mathbf{f}_k$, and a countably infinite set $X:=\{x_1,x_2,\dots,x_n,\dots\}$ whose elements are called \emph{variables}, we define the \emph{terms of type} $\tau$ and the \emph{alphabet} $\alf(u)$ of a term $u$ by the following induction:
\begin{itemize}
\item each variable $x\in X$ is a term and $\alf(x):=\{x\}$;
\item if $n_i=0$, then $\mathbf{f}_i$ is a term and $\alf(\mathbf{f}_i):=\varnothing$;
\item if $n_i>0$ and $u_1,\dots,u_{n_i}$ are terms, then the expression $\mathbf{f}_i(u_1,\dots,u_{n_i})$ is a term and $\alf(\mathbf{f}_i(u_1,\dots,u_{n_i})):=\cup_{j=1}^{n_i}\alf(u_j)$.
\end{itemize}
Let $T(X)$ stand for the set of all terms of type $\tau$ over $X$.  For each $i=1,\dots,k$, the map
\[
f_i\colon\underbrace{T(X)\times T(X)\times\dots\times T(X)}_{\text{$n_i$ times}}\to T(X)
\]
defined by ${f}_i(u_1,\dots,u_{n_i}):=\mathbf{f}_i(u_1,\dots,u_{n_i})$ is an $n_i$-ary operation on the set $T(X)$. Thus, we get the algebraic structure $\mathcal{T}(X)=(T(X),f_1,\dots,f_k)$ of type $\tau$.

An \emph{identity} of type $\tau$ is an expression $u=v$ with $u,v\in T(X)$. A structure $\mA$ of type $\tau$ \emph{satisfies} the identity $u=v$ if the images of the terms $u$ and $v$ coincide under all homomorphisms $\mathcal{T}(X)\to\mA$. Let $\Id\mA$ denote the set of all identities satisfied by $\mA$.

Given a set $\Sigma$ of identities of type $\tau$, we say that an identity $u=v$ \emph{follows} from $\Sigma$ or that $\Sigma$ \emph{implies} $u=v$ if every algebraic structure of type $\tau$ satisfying all identities of $\Sigma$ satisfies the identity $u=v$ as well. Birkhoff's completeness theorem of equational logic (see \cite[Theorem II.14.19]{BuSa81}) shows that this notion (which we have given a semantic definition) can be captured by
a natural set of inference rules. We need not going into more detail here because the completeness theorem is not utilized in this paper.

Given an algebraic structure $\mA$, an \emph{\ib} for $\mA$ is any set $\Sigma\subseteq\Id{\mA}$ such that every identity in $\Id{\mA}$ follows from $\Sigma$. A structure $\mA$ is said to be \emph{\fb} if it possesses a finite \ib; otherwise, $\mA$ is called \emph{\nfb}.

Given any set $\Sigma$ of identities of type $\tau$,  the class of all algebraic structures satisfying all identities from $\Sigma$ is called the \emph{variety defined by $\Sigma$}. It is easy to see that the satisfaction of an identity is inherited by forming direct products and taking \emph{divisors} (that is, homomorphic images of substructures) so that each variety is closed under these two operators. Varieties can be characterized by this closure property (the HSP-theorem; see \cite[Theorem II.11.9]{BuSa81}).

A variety is \emph{\fb} if it can be defined by a finite set of identities; otherwise it is \emph{\nfb}. Given an algebraic structure $\mA$, the variety defined by $\Id\mA$ is denoted by $\var\mA$ and called the \emph{variety generated by $\mA$}. By the very definition, $\mA$ and $\var\mA$ are simultaneously finitely or \nfb.

\subsection{Varieties defined by identities with $\le n$ variables}\label{subsec:boundedvariablenumber}

Given an algebraic structure $\mA$ and a positive integer $n$, we denote by $\Id^{(n)}\!\mA$ the set of all identities $u=v$ from $\Id\mA$ such that $\alf(u)\cup\alf(v)\subseteq\{x_1,x_2,\dots,x_n\}$. Let $\var^{(n)}\!\mA$ stand for the variety defined by $\Id^{(n)}\!\mA$. We need the two following properties of varieties of this form.

\begin{lemma}[{\!\mdseries\cite[Proposition IV.3.9]{Cohn}}]
\label{lem:n-generated}
Let $\mA$ and $\mB$ be algebraic structures of the same type, $n$ a positive integer. The structure $\mB$ belongs to the variety $\var^{(n)}\!\mA$ if and only if all its $n$-generated substructures lie in the variety $\var\mA$.
\end{lemma}

\begin{lemma}[{\mdseries Birkhoff's finite basis theorem, \cite[Theorem V.4.2]{BuSa81}}]
\label{lem:BirkhoffFBT}
For every finite algebraic structure $\mA$ and every positive integer $n$, the variety $\var^{(n)}\!\mA$ is finitely based.
\end{lemma}

\subsection{Inherently nonfinitely based structures}\label{subsec:infb}

A variety is said to be \emph{locally finite} if each of its finitely generated members is finite. A finite algebraic structure $\mA$ is called \emph{inherently \nfb} if $\mA$ is not contained in any finitely based locally finite variety. The variety generated by a finite structure is locally finite (this is an easy byproduct of the proof of the HSP-theorem; see \cite[Theorem II.10.16]{BuSa81}). Hence, a finite algebraic structure $\mA$ is \nfb{} if the variety $\var\mA$ contains an inherently \nfb\ structure.

Various versions of the following characterization of inherently \nfb\ structures are scattered over the literature (see \cite{McNSzeWi08} or, for a recent sample, \cite{ShWi23}), but, typically, they are not displayed as separate statements to which we could conveniently refer. Therefore, we provide a short proof for the sake of completeness.

\begin{proposition}
\label{prop:infb}
A finite algebraic structure $\mA$ is inherently \nfb\ if and only if for arbitrarily large positive integer $n$, there is an infinite, finitely generated structure $\mB_n$ of the same type as $\mA$, all of whose $n$-generated substructures lie in the variety $\var\mA$.
\end{proposition}

\begin{proof}
For the `only if' part, suppose that $\mA$ is inherently \nfb. By Lemma~\ref{lem:BirkhoffFBT}, the variety $\var^{(n)}\!\mA$ is finitely based for each positive integer $n$. Since $\mA$ belongs to $\var^{(n)}\!\mA$, this variety cannot be locally finite. Hence, for each $n$, there exists an infinite, finitely generated structure $\mB_n\in\var^{(n)}\!\mA$, and by Lemma~\ref{lem:n-generated}, all $n$-generated substructures of  $\mB_n$ lie in the variety $\var\mA$.

To prove the `if' part, we have to show that every \fb{} variety $\mathfrak V$ containing $\mA$ is not locally finite. Let $\Sigma$ be a finite \ib\ of $\mathfrak V$. Choose a positive integer $n$ such that all variables involved in the identities in $\Sigma$ occur in the set $\{x_1,\dots,x_n\}$ and there exists an infinite, finitely generated structure $\mB_n$ all of whose $n$-generated substructures lie in $\var\mA$. Take an arbitrary identity $u=v$ in $\Sigma$. The images of the terms $u$ and $v$ under an arbitrary homomorphism $\varphi\colon\mathcal{T}(X)\to\mB_n$ belong to the substructure $\mS$ of $\mB_n$ generated by the elements $\varphi(x_1),\dots,\varphi(x_n)$. By choice of $n$, the substructure $\mS$ belongs to the variety $\var\mA$. Then $\mS$ lies in the variety $\mathfrak V$ as $\mathfrak V$ contains $\mA$. The identity $u=v$ holds in $\mathfrak V$ whence it must be satisfied by $\mS$. This implies $\varphi(u)=\varphi(v)$.

We have thus verified that for each identity $u=v$ in $\Sigma$, the images of $u$ and $v$ coincide under an arbitrary homomorphism $\mathcal{T}(X)\to\mB_n$. Hence, $\mB_n$ satisfies all identities in $\Sigma$. Since $\Sigma$ is an \ib{} for $\mathfrak V$, we conclude that $\mB_n$ belongs to $\mathfrak V$. Since $\mB_n$ is infinite and finitely generated, the variety $\mathfrak V$ is not locally finite.
\end{proof}

\subsection{The \ais{} $\End(\mC_m)$ is inherently nonfinitely based for $m\ge 4$}\label{subsec:mge4}

Let $C_m$ denote the chain $0<1<\dots<m-1$. We consider the semilattice $\mC_m:=(C_m,+)$ with addition defined by $k+\ell:=\max\{k,\ell\}$ for all $k,\ell\in C_m$. The \ais{} $\End(\mC_m)$ of all endomorphisms of $\mC_m$ is the main object of the present paper.

The constant endomorphism $\omega$ that sends every element of the chain $C_m$ to its least element 0 is a neutral element of $(\End(\mC_m),+)$. The set $\End^0(\mC_m)$ consisting of all endomorphisms of $\mC_m$ that fix 0 forms a subsemiring in $\End(\mC_m)$. Clearly, $\omega$ is a multiplicative absorbing element in this subsemiring.

In~\cite{Dolinka09endo}, Dolinka considered $(\End^0(\mC_m),+,\cdot,\omega)$, for which he used the notation $\mathbf{O}_{m-1}$, as an algebraic structure of type $(2,2,0)$. Lemmas~5 and~6 in~\cite{Dolinka09endo} show that $(\End^0(\mC_m),+,\cdot,\omega)$ is inherently \nfb{} for each $m\ge 4$. Of course, this readily implies that for each $m\ge 4$, the structure $\mE_m:=(\End(\mC_m),+,\cdot,\omega)$ also is inherently \nfb, even though this corollary was not mentioned in~\cite{Dolinka09endo}\footnote{The apparent reason for this omission is that in~\cite{Dolinka09endo}, semirings were defined as structures $(R,+,\cdot,0)$ in which the constant 0 is both neutral for addition and absorbing for multiplication. While $(\End^0(\mC_m),+,\cdot,\omega)$ falls within the range of this definition, $(\End(\mC_m),+,\cdot,\omega)$ does not since $\omega$ is not absorbing for $(\End(\mC_m),\cdot)$.}.

Here, we consider \ais{}s as structures of type $(2,2)$ and show that the \ais{} $\End(\mC_m)$ with $m\ge 4$ remains inherently nonfinitely based in this setting too.

\begin{proposition}
\label{prop:infbmge4}
For each $m\ge 4$,  the \ais{} $\End(\mC_m)$ is inherently \nfb.
\end{proposition}

\begin{proof}
As mentioned above, the structure $\mE_m$ of type $(2,2,0)$ is inherently \nfb. By Proposition~\ref{prop:infb}, for arbitrarily large positive integer $n$, there is an infinite, finitely generated structure $\mB_n$ of type $(2,2,0)$, all of whose $n$-generated substructures lie in the variety $\var\mE_m$. Denote by $z$ the value of the nullary operation of $\mB_n$. We let $\mB'_n$ stand for the $(2,2)$-reduct of $\mB_n$, that is, the \ais{} obtained from $\mB_n$ by forgetting its nullary operation. The \ais{} $\mB'_n$ is infinite (it has the same carrier as $\mB_n$) and finitely generated (if a finite set $Y$ generates $\mB_n$, then the set $Y\cup\{z\}$ generates $\mB'_n$). Take an arbitrary $n$-generated subsemiring $\mS'$ of $\mB'_n$ and consider the subsemiring $\mS$ generated by $\mS'\cup\{z\}$. Since $\mS$ contains $z$, this subsemiring is closed under all operations of $\mB_n$, including the nullary one. Therefore, we can treat $\mS$ as a substructure of type $(2,2,0)$. Notice that $\mS$ is $n$-generated as such a substructure---the added generator $z$ does not count since the nullary operation  provides $z$ `for free'.

By the choice of $\mB_n$, its $n$-generated substructure $\mS$ belongs to the variety $\var\mE_m$. The HSP-theorem then says that $\mS$ is a homomorphic image of a substructure $\mT$ of a direct power of $\mE_m$. Passing to the $(2,2)$-reducts of all structures involved, we get that the \ais{} obtained from $\mS$ is a homomorphic image of the \ais{} obtained from $\mT$, and the latter \ais{} is a subsemiring of a direct power of the \ais{} $\End(\mC_m)$. Since the \ais{} $\mS'$ is a subsemiring of the $(2,2)$-reduct of $\mS$, we conclude that $\mS'$ belongs to the variety generated by $\End(\mC_m)$.

Thus, we see that $\mB'_n$ is an infinite, finitely generated \ais{} such that all its $n$-generated subsemirings lie in the variety $\var\End(\mC_m)$. Since $n$ is arbitrarily large, Proposition~\ref{prop:infb} applies, showing that $\End(\mC_m)$ is inherently \nfb.
\end{proof}

\begin{remark}
\label{rem:omit0}
The argument in the proof of Proposition~\ref{prop:infbmge4} is pretty general and can be used to show that omitting finitely many nullary operations cannot destroy the property of a finite algebraic structure to be inherently nonfinitely based. In contrast, the property is known to be sensitive to omitting unary operations. For instance, consider the 4-element semigroup
\[
B_0:=\langle e,f,a\mid e^2=e,\ f^2=f,\ eaf=a,\ ef=fe=0\rangle=\{e,f,a,0\}
\]
equipped with unary operations $x\mapsto x^*:=\begin{cases} e&\text{if $x=a$},\\x&\text{otherwise}\end{cases}$ and $x\mapsto x^+:=\begin{cases} f&\text{if $x=a$},\\ x&\text{otherwise}\end{cases}$.
Peter Jones has proved that the algebraic structure $(B_0,\cdot,{}^*,{}^+)$ of type $(2,1,1)$ is inherently nonfinitely based~\cite{Jones13} while each of the structures $(B_0,\cdot,{}^*)$ and $(B_0,\cdot,{}^+)$ of type $(2,1)$ is finitely based~\cite{Jones18}.
\end{remark}

\section{The \ais{}s $\mA_2^1$ and $\mB_2^1$}
\label{sec:a21}

The two 5-element semigroups
\begin{align*}
A_2&:=\langle e,a\mid eae=e^2=e,\ aea=a,\ a^2=0\rangle=\{e,a,ae,ea,0\},\\
B_2&:=\langle c,d\mid cdc=c,\ dcd=d,\ c^2=d^2=0\rangle=\{c,d,cd,dc,0\}
\end{align*}
play a distinguished role in both the structure theory of semigroups and the theory of semigroup varieties. The same can be said about the 6-element monoids $A_2^1$ and $B_2^1$ obtained by adjoining the identity element 1 to $A_2$ and $B_2$, respectively. For the present paper, it is important that each of these monoids can be equipped with addition making it an \ais.

\subsection{The \ais{} $\mB_2^1$}\label{subsec:b21}

Addition on the monoid $B_2^1$ (commonly known as the 6-\emph{element Brandt monoid}) is an instance of a general construction that is also needed in this paper, so we present this construction now.

Recall that elements $s,t$ of a semigroup $S$ are said to be \emph{inverses} of each other if $sts=s$ and $tst=t$. A~semigroup~$S$ is called \emph{regular} [respectively, \emph{inverse}] if every its element has an inverse [respectively, a unique inverse]. The unique inverse of an element $s$ of an inverse semigroup is denoted by $s^{-1}$. Inverse semigroups can therefore be thought of as algebraic structures of type (2,1) where the unary operation is defined by $s\mapsto s^{-1}$.

In every inverse semigroup $S$, the relation
\[
\le_{\mathrm{nat}}:= \{(s,t)\in S\times S\mid s=ss^{-1}t\}
\]
is a partial order compatible with both multiplication and inversion; see \cite[Section II.1]{Pet84} or \cite[pp. 21--23]{Law99}. This order is referred to as the \emph{natural partial order}. For a subset $H\subseteq S$, the infimum of $H$ with respect to $\le_{\mathrm{nat}}$ may not exist, but if $\inf H$ exists, then so do $\inf(sH)$ and $\inf(Hs)$ for every $s\in S$, and one has $\inf(sH) = s(\inf H)$ and $\inf(Hs) = (\inf H)s$ \cite[Proposition 1.22]{Sch73}; see also \cite[Proposition 19]{Law99}. Therefore, if the partially ordered set $(S,\le_{\mathrm{nat}})$ happens to be an inf-semilattice, then letting
\begin{equation}
\label{eq:inf}
s\,+_{\mathrm{nat}}\,t:=\inf\{s,t\}
\end{equation}
for all $s,t\in S$ makes $(S,+_{\mathrm{nat}},\cdot)$ be an \ais. We will need a sufficient condition for $(S,\le_{\mathrm{nat}})$ to be an inf-semilattice from the second-named author's note \cite{Vol21}. Here it is stated in the notation of the present paper.

\begin{lemma}[\!{\mdseries\cite[Lemma~2.1]{Vol21}}]
\label{lem:as-ais}
If an inverse semigroup $S$ satisfies for some $p$, the identity $x^p=x^{p+1}$, then $(S,\le_{\mathrm{nat}})$ is an $\inf$-semilattice and $s+_{\mathrm{nat}} t=(st^{-1})^ps$ for all $s,t\in S$.
\end{lemma}

By the definition of the 6-element Brandt monoid, $c^2=d^2=0$, and it is easy to see that $s^2=s$ for each $s\in B_2^1\setminus\{c,d\}$. Therefore, $B_2^1$ satisfies the identity $x^2=x^3$, and Lemma~\ref{lem:as-ais} applies. Thus, $(B_2^1,+_{\mathrm{nat}},\cdot)$ is an \ais; we denote it by $\mB_2^1$. It is known and easy to verify that $+_{\mathrm{nat}}$ is a unique addition making $B_2^1$ an \ais.

It is well known that on each semilattice $\mS=(S,+)$, the relation $\le$ defined by
\[
s\le s'\Leftrightarrow s+s'=s'\ \text{ for }\ s,s'\in S
\]
is a partial order on $S$ called the \emph{semilattice order}. For all elements $s,t\in S$, their sum $s+t$ coincides with the supremum of $s$ and $t$ for the semilattice order so that the order suffices to recover addition. When dealing with additive reducts of \ais{}s in the sequel, we always prefer drawing their Hasse diagrams to typing their Cayley tables, as diagrams encode addition more concisely as well as more visually.

The semilattice order of (the additive reduct of) the semiring $\mB_2^1$ is shown in Fig.~\ref{fig:b21}; observe that this order is opposite to the natural partial order of $B_2^1$ as an inverse semigroup.\\[-0.5ex]
\begin{figure}[tbh]
\centering
\begin{tikzpicture}[x=1cm,y=1.2cm]
\node at (0,-1)    (U)  {$1$};
\node at (-0.5,0) (EA)  {$cd$};
\node at (-1.5,0) (E) {$c$};
\node at (0.5,0)  (AE) {$dc$};
\node at (1.5,0)  (A)  {$d$};
\node at (0,1)    (Z)  {$0$};
\draw (A)  -- (Z);
\draw (EA) -- (Z);
\draw (AE) -- (Z);
\draw (E)  -- (Z);
\draw (EA) -- (U);
\draw (AE) -- (U);
\end{tikzpicture}
\caption{The semilattice order of the \ais{} $\mB_2^1$}\label{fig:b21}
\end{figure}

The FBP for the semiring $\mB_2^1$ was solved by the second-named author \cite{Vol21} and, independently, by Marcel Jackson,  Miaomiao Ren, and Xianzhong Zhao~\cite{JackRenZhao} who proved that this semiring is \nfb.

\subsection{The \ais{} $\mA_2^1$}\label{subsec:a21}

The monoid $A_2^1$ is regular, but not inverse since $e$ and $a$ are two different inverses of the element $e$. Thus, one cannot define addition on $A_2^1$ via \eqref{eq:inf}. Instead, we may embed $A_2^1$ into the multiplicative reduct of the semiring $\End(\mC_3)$ of all endomorphisms of the chain $0<1<2$ so that the image of $A_2^1$ in $\End(\mC_3)$ is closed under addition. There exist two embeddings of this sort: one where the endomorphisms corresponding to the elements of $A_2^1$ fix 2, and one where they fix 0. The first embedding is shown in Fig.~\ref{fig:A2asOPT}.
\begin{figure}[htb]
\centering
\begin{tikzpicture}
[scale=0.75]
\foreach \x in {0.75,2.25,3.5,5,6.25,7.75,9,10.5,11.75,13.25,14.5,16} \foreach \y in {0.5,1.5,2.5} \filldraw (\x,\y) circle (2pt);
	\draw (0.75,0.5) edge[-latex] (2.25,0.5);
	\draw (0.75,1.5) edge[-latex] (2.25,1.5);
    \draw (0.75,2.5) edge[-latex] (2.25,2.5);
	\draw (3.5,0.5) edge[-latex] (5,1.5);
	\draw (3.5,1.5) edge[-latex] (5,1.5);
    \draw (3.5,2.5) edge[-latex] (5,2.5);
	\draw (6.25,0.5) edge[-latex] (7.75,0.5);
	\draw (6.25,1.5) edge[-latex] (7.75,2.5);
    \draw (6.25,2.5) edge[-latex] (7.75,2.5);
	\draw (9,0.5) edge[-latex] (10.5,1.5);
	\draw (9,1.5) edge[-latex] (10.5,2.5);
    \draw (9,2.5) edge[-latex] (10.5,2.5);
	\draw (11.75,0.5) edge[-latex] (13.25,0.5);
	\draw (11.75,1.5) edge[-latex] (13.25,0.5);
    \draw (11.75,2.5) edge[-latex] (13.25,2.5);
    \draw (14.5,0.5) edge[-latex] (16,2.5);
	\draw (14.5,1.5) edge[-latex] (16,2.5);
    \draw (14.5,2.5) edge[-latex] (16,2.5);
\foreach \x in {0.5,2.5,3.25,5.25,6,8,8.75,10.75,11.5,13.5,14.25,16.25} \node at (\x,0.5) {\tiny 0};
\foreach \x in {0.5,2.5,3.25,5.25,6,8,8.75,10.75,11.5,13.5,14.25,16.25} \node at (\x,1.5) {\tiny 1};
\foreach \x in {0.5,2.5,3.25,5.25,6,8,8.75,10.75,11.5,13.5,14.25,16.25} \node at (\x,2.5) {\tiny 2};
\node at (1.5,-0.15)   {$1$};
\node at (4.25,-0.15)  {$ea$};
\node at (7,-0.15)     {$ae$};
\node at (9.75,-0.15)  {$a$};
\node at (12.5,-0.15)  {$e$};
\node at (15.25,-0.15) {$0$};
\end{tikzpicture}
\caption{Representing the elements of ${A}^1_2$ by the endomorphisms of~the chain $0<1<2$ that fix 2}\label{fig:A2asOPT}
\end{figure}
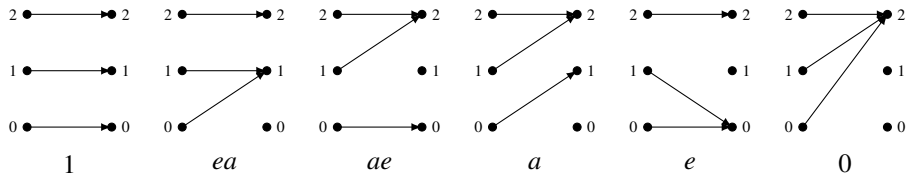

The endomorphisms of~the chain $0<1<2$ that fix 2 form a subsemiring in the semiring $\End(\mC_3)$, and this subsemiring's addition induces addition in $A_2^1$. We denote the resulting \ais{} by $\mA_2^1$. The semilattice order of $\mA_2^1$ is shown in Fig.~\ref{fig:a21}. It corresponds to the point-wise order for the endomorphisms of~the chain $0<1<2$ in which $\alpha\le\beta$ means that $x\alpha\le x\beta$ for each $x=0,1,2$.
\begin{figure}[ht]
\centering
\begin{tikzpicture}
[scale=0.8]
\node at (0,-1)   (E)  {$e$};
\node at (0,0)    (U)  {$1$};
\node at (-1.5,1) (EA) {$ea$};
\node at (1.5,1)  (AE) {$ae$};
\node at (0,2)    (A)  {$a$};
\node at (0,3)    (Z)  {$0$};
\draw (A)   -- (Z);
\draw (EA)  -- (A);
\draw (AE)  -- (A);
\draw (EA)  -- (U);
\draw (AE)  -- (U);
\draw (E)   -- (U);
\end{tikzpicture}
\caption{The semilattice order of the \ais{} $\mA_2^1$}\label{fig:a21}
\end{figure}

The image of the alternative embedding of $A_2^1$ in $\End(\mC_3)$ consists of the endomorphisms of~the chain $0<1<2$ that fix 0. This image also is a subsemiring in $\End(\mC_3)$, being nothing but the \ais{} $\End^0(\mC_3)$, aka $\mathbf{O}_{2}$ in Dolinka's notation~\cite{Dolinka09endo}. The semilattice order of $\End^0(\mC_3)$ is the point-wise order for the endomorphisms of~the chain $0<1<2$, and it can be easily verified to be dual to that of $\mA_2^1$. Since the semilattice shown in Fig.~\ref{fig:a21} is actually a self-dual lattice, we see that not only the multiplicative but also the additive reducts of the \ais{}s $\mA_2^1$ and $\End^0(\mC_3)$ are isomorphic. However, the \ais{}s themselves are not isomorphic! Say, $\mA_2^1$ has an element that is absorbing for both addition and multiplication while $\End^0(\mC_3)$ has no such element. Moreover, the identities satisfied by these \ais{}s are different: for instance, $\mA_2^1$ satisfies the identity $x+x^2=x^2$ that fails in $\End^0(\mC_3)$, and $\End^0(\mC_3)$ satisfies the identity $x+x^2=x$ that fails in $\mA_2^1$.

It has been the FBP for the semiring $\mA_2^1$ that gave impetus to the present paper. It follows from Mark Sapir's classification of inherently \nfb{} semigroups \cite{Sa87a,Sa87b} that six is the minimum size of such a semigroup, and the monoids $A_2^1$ and $B_2^1$ are the only inherently \nfb{} semigroups of this size. Since the \ais{} $\mB_2^1$ has been shown to be \nfb, the natural question of whether or not $\mA_2^1$ behaves the same way has become rather intriguing. To answer this question, a novel approach has been developed, and the new tool has allowed us to solve the FBP for $\End(\mC_3)$ too. However, the FBP for $\End^0(\mC_3)$ still remains open.

\subsection{Outline of the proof that $\mA_2^1$ and $\End(\mC_3)$ are \nfb}\label{subsec:plan} Now, with the necessary notions and notation being introduced, we are in a position to explain how our proof proceeds.

To show that the \ais{}s $\mA_2^1$ and $\End(\mC_3)$ are \nfb, we construct an infinite series of finite inverse semigroups $(S_n,\cdot,{}^{-1})$, $n\ge2$, each of which satisfies the identity $x^2=x^3$ and has the two following properties:
\begin{itemize}
  \item[(i)] the semigroup $(S_n,\cdot)$ does not lie in the semigroup variety generated by the multiplicative reduct of $\End(\mC_3)$;
  \item[(ii)] all $(n-1)$-generated inverse subsemigroups of $(S_n,\cdot,{}^{-1})$ belong to the inverse semigroup variety generated by $(B_2^1,\cdot,{}^{-1})$.
\end{itemize}
By Lemma~\ref{lem:as-ais}, each inverse semigroup $(S_n,\cdot,{}^{-1})$ converts into the \ais{} $(S_n,+_{\mathrm{nat}},\cdot)$ via \eqref{eq:inf}. Property (i) ensures that none of these semirings lies in the variety $\var\End(\mC_3)$, and from (ii), one can deduce that all $(n-1)$-generated subsemirings of $(S_n,+_{\mathrm{nat}},\cdot)$ belong to the variety $\var\mB_2^1$. A standard reasoning (detailed in Section~\ref{sec:main} below) then shows that any \ais{} $\mR$ such that $\mB_2^1\in\var\mR$ and $\mR\in\var\End(\mC_3)$ is \nfb. (In order to apply the outlined argument to $\mA_2^1$ and $\End(\mC_3)$, one has to verify the containment $\mB_2^1\in\var\mA_2^1$; we will do this in the next subsection.)

A similar overall tactic has been utilized in our earlier articles \cite{GuVo23a,GuVo23b} on the FBP for \ais{}s, but there, we have re-used a series of inverse semigroups constructed by Ji\v{r}\'\i{} Ka\softd{}ourek in his  breakthrough paper~\cite{Kad03} on the FBP for finite inverse semigroups. Property (ii) for that series was established in~\cite{Kad03} so that the most complicated part of the job had already been handled. The major difference between the present paper and \cite{GuVo23a,GuVo23b} is that we cannot employ Ka\softd{}ourek's semigroups anymore. So, we have to introduce a new series $(S_n,\cdot,{}^{-1})$ and verify (ii) for it. This will be done in Section~\ref{sec:semigroups-sn}, after recalling, in Section~\ref{sec:kadourek}, a result from~\cite{Kad91} that we need for establishing property (ii) for our series.

\subsection{$\mB_2^1$ lies in the variety $\var\mA_2^1$}\label{subsec:a21b21relation}

Recall that an \ais\ $\mS$ \emph{divides} another \ais\ $\mT$ if $\mS$ is a homomorphic image of a subsemiring of $\mT$.

\begin{proposition}
\label{prop:b21-in-a21}
The \ais\ $\mB_2^1$ divides the direct square $\mA_2^1\times \mA_2^1$.
\end{proposition}

\begin{proof}
Consider the subsemiring $\mB$ of $\mA_2^1\times \mA_2^1$ generated by the pairs $(1,1)$, $(e,a)$, and $(a,e)$. Since 0 is an absorbing element for both multiplication and addition in the \ais\ $\mA_2^1$, the set $\mN:=\{(x,y)\in\mB\mid x=0\ \text{ or }\ y=0 \}$ forms an ideal in both the multiplicative and additive reducts of $\mB$. We can therefore pass to the Rees quotient semiring $\mB/\mN$. The non-zero elements of $\mB/\mN$ are in a one-to-one correspondence with pairs $(x,y)\in\mB$ such that neither $x$ nor $y$ is 0. A direct calculation shows that there are twelve such pairs:
\begin{itemize}
  \item the generators $(1,1)$, $(e,a)$, $(a,e)$;
  \item $(1,a)=(1,1)+(e,a)$, \  $(a,1)=(1,1)+(a,e)$, \ $(a,a)=(e,a)+(a,e)$;
  \item $(ea,ae)=(e,a)(a,e)$, \ $(ae,ea)=(a,e)(e,a)$;
  \item $(ae,a)=(ae,ea)+(e,a)$, \ $(a,ae)=(ea,ae)+(a,e)$;
  \item $(ea,a)=(ea,ae)+(e,a)$, \ $(a,ea)=(ae,ea)+(a,e)$.
\end{itemize}
The semilattice order on the semiring $\mB/\mN$ is shown in Fig.~\ref{fig:bovern}.
\begin{figure}[hbt]
\centering
\begin{tikzpicture}[x=1cm,y=1.2cm]
\node at (0,0)  (E)   {$(1,1)$};
\node at (-4,0) (C)   {$(a,e)$};
\node at (4,0)  (D)   {$(e,a)$};
\node at (-4,1) (C1)  {$(a,1)$};
\node at (-2,1) (CD)  {$(ea,ae)$};
\node at (2,1)  (DC)  {$(ae,ea)$};
\node at (4,1)  (D1)  {$(1,a)$};
\node at (-4,2) (C2)  {$(a,ae)$};
\node at (-2,2) (CD2)  {$(a,ea)$};
\node at (2,2)  (DC2) {$(ea,a)$};
\node at (4,2)  (D2)  {$(ae,a)$};
\node at (0,4)  (Z)   {$0$};
\node at (0,3)  (AA)  {$(a,a)$};
\draw (E)  -- (C1);
\draw (E)  -- (CD);
\draw (E)  -- (DC);
\draw (E)  -- (D1);
\draw (C)  -- (C1);
\draw (D)  -- (D1);
\draw (AA)  -- (Z);
\draw (C1)  -- (C2);
\draw (C1)  -- (CD2);
\draw (D1)  -- (D2);
\draw (D1)  -- (DC2);
\draw (CD)  -- (C2);
\draw (CD)  -- (DC2);
\draw (DC)  -- (D2);
\draw (DC)  -- (CD2);
\draw (AA)  -- (C2);
\draw (AA)  -- (CD2);
\draw (AA)  -- (D2);
\draw (AA)  -- (DC2);
\draw [red] plot [smooth] coordinates { (-4.8,-0.1) (-2,1.5) (2,1.5) (4.8,-0.1)};
\end{tikzpicture}
\caption{The order of the \ais{} $\mB/\mN$}\label{fig:bovern}
\end{figure}

It is routine to verify that the 8-element set
\[
\{0,(1,a),(a,1),(ae,a),(a,ae),(ea,a),(a,ea),(a,a)\}
\]
forms an ideal in both the multiplicative and additive reducts of $\mB/\mN$ (in Fig.~\ref{fig:bovern}, the elements of this set lie above the red curve). The Rees quotient of $\mB/\mN$ over this ideal is easily seen to be isomorphic to the \ais\ $\mB_2^1$.
\end{proof}

Proposition~\ref{prop:b21-in-a21} readily yields the result we need.
\begin{corollary}\label{cor:b21-in-a21}
The \ais{} $\mB_2^1$ lies in the variety $\var\mA_2^1$
\end{corollary}

\section{Ka\softd{}ourek's criterion}
\label{sec:kadourek}

A semigroup $S$ is called \emph{combinatorial} if all subgroups of $S$ are trivial. Here we recall Ka\softd{}ourek's criterion~\cite{Kad91} for a finite combinatorial inverse semigroup to satisfy all identities of the $6$-element Brandt monoid $(B_2^1,\cdot,{}^{-1})$.

We need some definitions and notation. We denote by $E(S)$ the set of all idempotents of $S$. If $X$ is a subset of $S$, we let $E(X) := X \cap E(S)$. For $e,f \in E(S)$, we write $e \le f$ if $ef = fe = e$. Now let $(S,\cdot,{}^{-1})$ be a finite combinatorial inverse semigroup and let $e,g\in E(S)$ be such that $e\le g$. For any idempotent $h$ in the $\mathscr D$-class $D_g$ of the idempotent $g$, there exists a unique element $a \in D_g$ with $g = aa^{-1}$, $h = a^{-1}a = a^{-1}ga$. This then determines a map $\pi_{g,e}\colon E(D_g)\to E(D_e)$ by
\[
h\pi_{g,e} := a^{-1}ea = (ea)^{-1}(ea).
\]
This map produces an idempotent $f\le h$ whose position with respect to the ``anchor'' idempotent $e$ of $D_e$ precisely corresponds to the position $h$ has had with respect to the ``anchor'' $g$ of its $\mathscr D$-class, see Fig.~\ref{fig:projection} where horizontal and vertical arrows symbolize multiplying respectively on the right and on the left.
\begin{figure}[ht]
\begin{center}
\unitlength .7mm
\begin{picture}(120,75)
\thinlines
\gasset{ExtNL=y,Nw=1,Nh=1,Nfill=y,fillgray=0.5,NLdist=1}
\node[NLangle=180](A)(10,5){$a^{-1}e$}
\node[NLangle=0,NLdist=2](B)(40,5){$f=a^{-1}ea$}
\node[NLangle=180](C)(10,35){$e$}
\node[NLangle=0](D)(40,35){$ea$}
\drawedge[ELside=r](A,B){$a$}
\drawedge(C,D){$a$}
\drawedge[ELside=r](C,A){$a^{-1}$}
\drawedge(D,B){$a^{-1}$}
\node[NLangle=180](E)(80,45){$a^{-1}g$}
\node[NLangle=0](F)(110,45){$h=a^{-1}ga$}
\node[NLangle=180,NLdist=2](G)(80,75){$g$}
\node[NLangle=0](H)(110,75){$ga$}
\drawedge[ELside=r](E,F){$a$}
\drawedge(G,H){$a$}
\drawedge[ELside=r](G,E){$a^{-1}$}
\drawedge(H,F){$a^{-1}$}
\put(21,18){\Large$D_e$}
\put(91,58){\Large$D_g$}
\gasset{dash={2 1.4}{0}}
\drawedge[ELside=r](G,C){$\pi_{g,e}$}
\drawedge(F,B){$\pi_{g,e}$}
\end{picture}
\caption{The map $\pi_{g,e}$}
\label{fig:projection}
\end{center}
\end{figure}
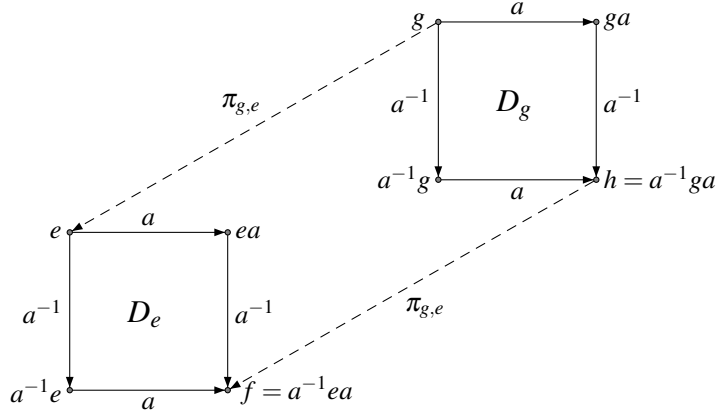

The set $S/\mathscr D$ of all $\mathscr D$-classes of the semigroup $S$ has a natural partial order defined as follows: for $X,Y\in S/\mathscr D$,
\[
Y \le X\ \text{ if and only if }\ e \le f\ \text{ for some }\ e \in E(Y),\ f \in E(X).
\]
Let $X,Y$ be $\mathscr D$-classes of $S$ and $Y\le X$. We define two symmetric relations on $E(Y)$. The first one---the \textit{projection relation}---is defined by
\[
\pi(X,Y) := \left\{(h_1\pi_{g,e},h_2\pi_{g,e}) \mid g,h_1,h_2 \in E(X),\ e \in E(Y),\ e \le g\right\}.
\]
The second relation contains all pairs of idempotents in $E(Y)$ which have a common upper bound in $E(X)$:
\[
\rho(X,Y) :=\left\{(f_1,f_2) \in E(X) \mid f_1,f_2 \le g\ \text{ for some }\ g \in E(Y)\right\}.
\]
For any $Y \in S/\mathscr D$, we write $[Y)_S := \{XY\in S/\mathscr D \mid Y \le X\}$. A (possibly empty) subset $\mathcal K$ of $S/\mathscr D$ is called a \textit{filter} if $[Z)_S\subseteq\mathcal K$ for all $Z\in\mathcal K$. Given $Y \in S/\mathscr D$ and a filter $\mathcal K \subseteq [Y)_S$, we denote by $\tau(\mathcal K ,Y)$ the equivalence relation being the transitive closure of the relation
\[
\bigcup\left\{\pi(X_1,Y) \mid X_1 \in \mathcal K\right\}\cup \bigcup\left\{\rho(X_2,Y) \mid X_2 \in [Y)_S \setminus\mathcal K\right\}.
\]
We say that the filter $\mathcal K$ \textit{separates} idempotents $e,f \in E(Y)$ if $(e,f)\notin\tau(\mathcal K,Y)$.

Now we can state Ka\softd{}ourek's result.

\begin{proposition}[\!\!{\mdseries\cite[Theorem 2.3]{Kad91}}]
\label{prop:kadourek}
A finite combinatorial inverse semigroup $(S,\cdot,{}^{-1})$ lies in the inverse semigroup variety generated by the $6$-element Brandt monoid $(B_2^1,\cdot,{}^{-1})$ if and only if the following condition holds:
\begin{itemize}
\item[\textup{($\ast$)}] for any $\mathscr D$-classes $X,Y$ of $S$ and for any $e,f \in E(Y)$ such that $e \le g$ and $f \nleq g$ for some $g \in E(X)$, there exists a filter $\mathcal K \subseteq [Y)_S$ which does not include $X$ and separates $e$ from $f$.
\end{itemize}
\end{proposition}

\section{The semigroups $S_n$}
\label{sec:semigroups-sn}

For all $n\ge 2$, let $(S_n,\cdot,{}^{-1})$ stand for the inverse semigroup of partial one-to-one transformations on the set $\{0,1,\dots,3n+2\}$ generated by the following $n+1$ transformations:
\[
\chi:=\begin{pmatrix} n&n+1\\2n+1&2n+2\end{pmatrix}, \ \ \chi_i:=\begin{pmatrix} i-1&n+1+i&2n+1+i\\i&n+i&2n+2+i\end{pmatrix},\ i=1,2,\dots,n.
\]
The action of the generators of the semigroup $(S_n,\cdot,{}^{-1})$ is shown in Fig.~\ref{fig:semigroups-sn}.
\begin{figure}[ht]
\begin{center}
\unitlength 1.35mm
\begin{picture}(40,54)
\gasset{Nw=5,Nh=5,Nmr=3}
\node(u0)(0,52){\footnotesize 0}
\node(u1)(0,42){\footnotesize 1}
\node(u2)(0,32){\footnotesize 2}
\node[Nframe=n](u3)(0,22){$\vdots$}
\node(un-1)(0,12){\footnotesize n-1}
\node(un)(0,2){\footnotesize n}
\node(un+1)(20,52){\footnotesize n+1}
\node(un+2)(20,42){\footnotesize n+2}
\node(un+3)(20,32){\footnotesize n+3}
\node[Nframe=n](un+4)(20,22){$\vdots$}
\node(u2n)(20,12){\footnotesize 2n}
\node(u2n+1)(20,2){\footnotesize 2n+1}
\node(u2n+2)(40,52){\footnotesize 2n+2}
\node(u2n+3)(40,42){\footnotesize 2n+3}
\node(u2n+4)(40,32){\footnotesize 2n+4}
\node[Nframe=n](u2n+5)(40,22){$\vdots$}
\node(u3n+1)(40,12){\footnotesize 3n+1}
\node(u3n+2)(40,2){\footnotesize 3n+2}
\drawedge(u0,u1){$\chi_1$}
\drawedge(u1,u2){$\chi_2$}
\drawedge(u2,u3){$\chi_3$}
\drawedge(u3,un-1){$\chi_{n-1}$}
\drawedge(un-1,un){$\chi_n$}
\drawedge(un,u2n+1){$\chi$}
\drawedge(u2n+1,u2n){$\chi_n$}
\drawedge(u2n,un+4){$\chi_{n-1}$}
\drawedge(un+4,un+3){$\chi_3$}
\drawedge(un+3,un+2){$\chi_2$}
\drawedge(un+2,un+1){$\chi_1$}
\drawedge(un+1,u2n+2){$\chi$}
\drawedge(u2n+2,u2n+3){$\chi_1$}
\drawedge(u2n+3,u2n+4){$\chi_2$}
\drawedge(u2n+4,u2n+5){$\chi_3$}
\drawedge(u2n+5,u3n+1){$\chi_{n-1}$}
\drawedge(u3n+1,u3n+2){$\chi_n$}
\end{picture}
\caption{The action of the generators of the inverse semigroup $(S_n,\cdot,{}^{-1})$}\label{fig:semigroups-sn}
\end{center}
\end{figure}
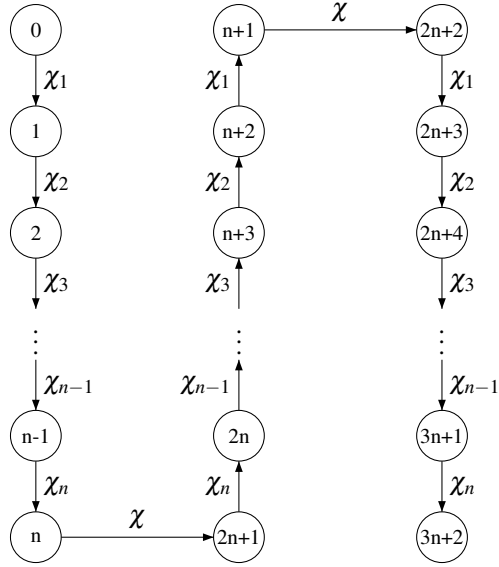

We need to compute computing all transformations in $S_n$. Clearly, $S_n$ contains the nowhere defined transformation which we denote by 0. Further, the following formulas hold:
\begin{equation}
\label{eq:formulas}
\begin{aligned}
\chi_i\chi_j={}&
\begin{cases}
\begin{pmatrix} i-1&2n+1+i\\i+1&2n+3+i\end{pmatrix}&\text{if }i=j-1,\\
\begin{pmatrix} n+1+i\\n-1+i\end{pmatrix}&\text{if }i=j+1,\\
0&\text{otherwise;}
\end{cases}\\
\chi_i\chi_j^{-1}={}&
\begin{cases}
\begin{pmatrix} i-1&n+1+i&2n+1+i\\i-1&n+1+i&2n+1+i\end{pmatrix}&\text{if }i=j,\\
0&\text{otherwise;}
\end{cases}\\
\chi_i^{-1}\chi_j={}&
\begin{cases}
\begin{pmatrix} i&n+i&2n+2+i\\i&n+i&2n+2+i\end{pmatrix}&\text{if }i=j,\\
0&\text{otherwise;}
\end{cases}
\end{aligned}
\end{equation}
\begin{equation}
\label{eq:formulas2}
\begin{aligned}
\chi\chi_i={}&
\begin{cases}
\begin{pmatrix} n+1\\2n+3\end{pmatrix}&\text{if }i=1,\\
\begin{pmatrix} n\\2n\end{pmatrix}&\text{if }i=n,\\
0&\text{otherwise;}
\end{cases}
\\
\chi_i\chi={}&
\begin{cases}
\begin{pmatrix} n+2\\2n+2\end{pmatrix}&\text{if }i=1,\\
\begin{pmatrix} n-1\\2n+1\end{pmatrix}&\text{if }i=n,\\
0&\text{otherwise;}
\end{cases}
\\
\chi\chi_j^{-1}&{}=\chi_j^{-1}\chi=0,\ j=1,2,\dots,n.
\end{aligned}
\end{equation}
For any $i,j\in\{0,1,\dots,n\}$ and $\ell,r\in\{0,1,\dots,3n+2\}$, we put
\[
\zeta_{ij}:=\begin{pmatrix} i&2n+2+i\\j&2n+2+j\end{pmatrix}\ \text{ and }\ \eta_{\ell r}:=\begin{pmatrix} \ell\\r\end{pmatrix}.
\]
It is easy to see that $\zeta_{ij}$ and $\eta_{\ell r}$ can be represented as products of elements of the form $\chi$, $\chi^{-1}$, $\chi_p$ and $\chi_p^{-1}$.
Hence $\zeta_{ij},\eta_{\ell r}\in S_n$.
Now let
\begin{align*}
B_i&:=\{\chi_i,\chi_i^{-1},\chi_i\chi_i^{-1},\chi_i^{-1}\chi_i\},\ i=1,2,\dots,n,\\
C&:=\{\zeta_{ij}\mid 1\le i,j\le n\},\\
D&:=\{\eta_{ij}\mid 0\le i,j\le 3n\},\\
E&:=\{\chi,\chi^{-1},\chi\chi^{-1},\chi^{-1}\chi\}.
\end{align*}
Using \eqref{eq:formulas} and \eqref{eq:formulas2}, it is easy to verify that
\begin{equation}\label{eq:rank decomposition}
S_n=\left(\bigcup_{i=1}^n B_i\right)\cup C\cup D\cup E\cup\{0\}.
\end{equation}
Recall that the \emph{rank} of a transformation of a finite set is the size of its image. One sees that $\bigcup_{i=1}^n B_i$ is the set of all transformations of rank 3 in $S_n$ while $C\cup E$ and $D$ are the sets of all transformations of rank 2 and rank~1, respectively.

Now we verify that the decomposition \eqref{eq:rank decomposition} is actually the partition of $S_n$ into $\mathscr D$-classes. Indeed, $\{0\}$ is obviously a $\mathscr D$-class. Since $\chi_i$ and $\chi_i^{-1}$ are inverses of each other, $B_i$ must be contained in a $\mathscr D$-class of $S_n$. It follows from~\eqref{eq:rank decomposition} that if $\pi\in S\setminus B_i$, then $\pi\in B_j$ for some $j\ne i$ or $\pi$ is of rank less than $3$. In any case, $\pi$ cannot be $\mathscr D$-related with elements of $B_i$. Thus, $B_i$ forms a $\mathscr D$-class of $S_n$. It follows from formulas~\eqref{eq:formulas2} that no element of $E$ is $\mathscr D$-related with any $\pi\notin E$ of rank 2. Evidently, all elements of $E$ are $\mathscr D$-related. Hence $E$ forms a $\mathscr D$-class of $S_n$.  Further, formulas~\eqref{eq:formulas} and~\eqref{eq:formulas2} imply that if $\pi\in S_n\setminus E$ and $\pi$ is of rank 2, then the domain and the image of $\pi$ have the form $\{i,2n+i\}$ for some $0\le i\le n$. Since $\zeta_{ij}=\zeta_{i\ell}\zeta_{\ell j}$ for any $0\le i,j,\ell\le n$, this implies that $C$ forms a $\mathscr D$-class of $S_n$. By a similar argument, we can show that $D$ is a $\mathscr D$-class of $S_n$ as well. 

The partially ordered set $S_n/\mathscr D$ has the form shown in Fig.~\ref{fig:D-classes}.
\begin{figure}[ht]
\begin{center}
\unitlength 1.3mm
\begin{picture}(95,81)
\gasset{Nw=10,Nh=5,Nmr=0}
\node[Nframe=n,Nw=5](B1)(2,80){$B_1:$}
\node(B1-1-1)(10,80){$\chi_1\chi_1^{-1}$}
\node(B1-2-1)(10,75){$\chi_1^{-1}$}
\node(B1-1-2)(20,80){$\chi_1$}
\node(B1-2-2)(20,75){$\chi_1^{-1}\chi_1$}
\node[Nframe=n,Nw=0,Nh=0](B1-b)(15,72.5){}

\node[Nframe=n,Nw=5](B2)(32,80){$B_2:$}
\node(B2-1-1)(40,80){$\chi_2\chi_2^{-1}$}
\node(B2-2-1)(40,75){$\chi_2^{-1}$}
\node(B2-1-2)(50,80){$\chi_2$}
\node(B2-2-2)(50,75){$\chi_2^{-1}\chi_2$}
\node[Nframe=n,Nw=0,Nh=0](B2-b)(45,72.5){}

\node[Nframe=n,Nw=5](B2)(61.5,77.5){$\cdots$}

\node[Nframe=n,Nw=5](Bn)(67,80){$B_n:$}
\node(B2-1-1)(75,80){$\chi_n\chi_n^{-1}$}
\node(B2-2-1)(75,75){$\chi_n^{-1}$}
\node(B2-1-2)(85,80){$\chi_n$}
\node(B2-2-2)(85,75){$\chi_n^{-1}\chi_n$}
\node[Nframe=n,Nw=0,Nh=0](B3-b)(80,72.5){}

\node[Nframe=n,Nw=5](C)(2,60){$C:$}
\node(C-1-1)(10,60){$\zeta_{00}$}
\node(C-2-1)(10,55){$\zeta_{10}$}
\node(C-3-1)(10,50){$\cdots$}
\node(C-4-1)(10,45){$\zeta_{n0}$}
\node(C-1-2)(20,60){$\zeta_{01}$}
\node(C-2-2)(20,55){$\zeta_{11}$}
\node(C-3-2)(20,50){$\cdots$}
\node(C-4-2)(20,45){$\zeta_{n1}$}
\node[Nw=30](C-1-3)(40,60){$\cdots$}
\node[Nw=30](C-2-3)(40,55){$\cdots$}
\node[Nw=30](C-3-3)(40,50){$\cdots$}
\node[Nw=30](C-4-3)(40,45){$\cdots$}
\node(C-1-4)(60,60){$\zeta_{0n}$}
\node(C-2-4)(60,55){$\zeta_{1n}$}
\node(C-3-4)(60,50){$\cdots$}
\node(C-4-4)(60,45){$\zeta_{nn}$}
\node[Nframe=n,Nw=0,Nh=0](C-t)(35,62.5){}
\node[Nframe=n,Nw=0,Nh=0](C-b)(35,42.5){}

\node[Nframe=n,Nw=5](E)(72,50){$E:$}
\node(E-1-1)(80,50){$\chi\chi^{-1}$}
\node(E-2-1)(80,45){$\chi^{-1}$}
\node(E-1-2)(90,50){$\chi$}
\node(E-2-2)(90,45){$\chi^{-1}\chi$}
\node[Nframe=n,Nw=0,Nh=0](E-b)(85,42.5){}

\node[Nframe=n,Nw=5](D)(12,30){$D:$}
\node(D-1-1)(20,30){$\eta_{00}$}
\node(D-2-1)(20,25){$\eta_{10}$}
\node(D-3-1)(20,20){$\cdots$}
\node(D-4-1)(20,15){$\eta_{3n+2,0}$}
\node(D-1-2)(30,30){$\eta_{01}$}
\node(D-2-2)(30,25){$\eta_{11}$}
\node(D-3-2)(30,20){$\cdots$}
\node(D-4-2)(30,15){$\eta_{3n+2,1}$}
\node[Nw=30](D-1-3)(50,30){$\cdots$}
\node[Nw=30](D-2-3)(50,25){$\cdots$}
\node[Nw=30](D-3-3)(50,20){$\cdots$}
\node[Nw=30](D-4-3)(50,15){$\cdots$}
\node[Nw=12](D-1-4)(71,30){$\eta_{0,3n+2}$}
\node[Nw=12](D-2-4)(71,25){$\eta_{1,3n+2}$}
\node[Nw=12](D-3-4)(71,20){$\cdots$}
\node[Nw=12](D-4-4)(71,15){$\eta_{3n+2,3n+2}$}
\node[Nframe=n,Nw=0,Nh=0](D-t)(45,32.5){}
\node[Nframe=n,Nw=0,Nh=0](D-b)(45,12.5){}

\node(0-1-1)(45,0){$0$}
\node[Nframe=n,Nw=0,Nh=0](0-t)(45,2.5){}

\drawedge[AHnb=0](B1-b,C-t){}
\drawedge[AHnb=0](B2-b,C-t){}
\drawedge[AHnb=0](B3-b,C-t){}
\drawedge[AHnb=0](C-b,D-t){}
\drawedge[AHnb=0](E-b,D-t){}
\drawedge[AHnb=0](D-b,0-t){}
\end{picture}
\caption{The partially ordered set $S_n/\mathscr D$}
\label{fig:D-classes}
\end{center}
\end{figure}
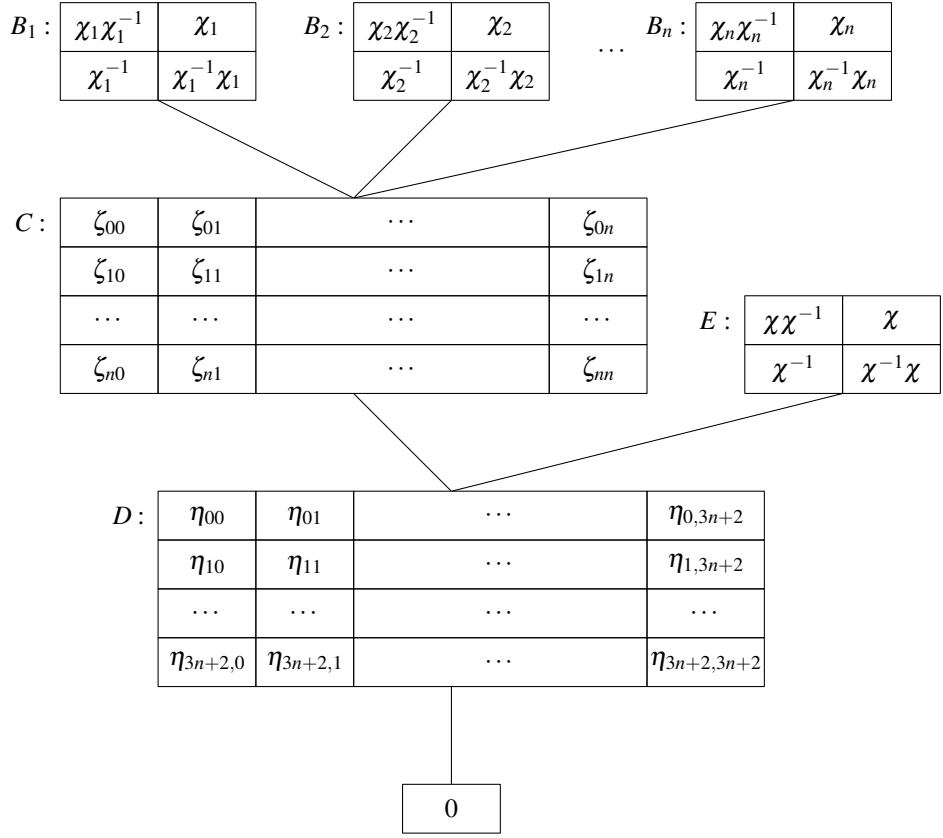

We turn to establish a crucial property of the inverse semigroup $(S_n,\cdot,{}^{-1})$ stated in the next proposition.

\begin{proposition}
\label{prop:Sn1}
Each inverse subsemigroup of $(S_n,\cdot,{}^{-1})$ generated by less than $n$ elements lies in the inverse semigroup variety generated by the $6$-element Brandt monoid $(B_2^1,\cdot,{}^{-1})$.
\end{proposition}

\begin{proof}
Let $T_n(k):=S_n\setminus B_k$, $k=1,2,\dots,n$. Evidently, $(T_n(k),\cdot,{}^{-1})$ forms an inverse subsemigroup of $(S_n,\cdot,{}^{-1})$ and any inverse subsemigroup of $(S_n,\cdot,{}^{-1})$ generated by less than $n$ elements is, in fact, a subsemigroup of $(T_n(k),\cdot,{}^{-1})$ for some $k$. Thus, it suffices to show that the inverse semigroup $(T_n(k),\cdot,{}^{-1})$ belongs to the inverse semigroup variety generated by $(B_2^1,\cdot,{}^{-1})$. For this, we need to verify the condition ($\ast$) of Proposition~\ref{prop:kadourek} for all $\mathscr D$-classes of the inverse semigroup $(T_n(k),\cdot,{}^{-1})$. Notice that the $\mathscr D$-classes of $(T_n(k),\cdot,{}^{-1})$ are $B_i$ with $i\ne k$, $C$, $D$, $E$, and $\{0\}$.

We fix an arbitrary pair $X$, $Y$ of $\mathscr D$-classes in $T_n(k)/\mathscr D$ such that there exist idempotents $\mu,\nu\in Y$, $\gamma \in X$ satisfying $\mu \le \gamma$ and $\nu\nleq \gamma$. Clearly, $\mu\ne \nu$  whence $Y\ne\{0\}$.  To verify the condition ($\ast$), we should exhibit a filter $\mathcal K\subseteq [Y)_{T_n(k)}$ that excludes $X$ and separates $\mu$ from $\nu$. Of course, the filter will depend on $X$ and $Y$ for which several options exist, so a case-by-case analysis is inevitable.

Before we proceed with the analysis, we introduce some notation and make some preliminary calculations. Given a set $A$, denote by $\nabla(A)$ [respectively, $\Delta(A)$] the universal [the equality] relation on $A$; we write $\nabla(a_1,a_2,\dots,a_s)$ [respectively, $\Delta(a_1,a_2,\dots,a_s)$] rather than $\nabla(\{a_1,a_2,\dots,a_s\})$ [respectively, $\Delta(\{a_1,a_2,\dots,a_s\})$]. For brevity, put $\eta_i:=\eta_{ii}$ for $i=0,1,\dots,3n+2$ and $\zeta_j:=\zeta_{jj}$ for $j=0,1,\dots,n$; these transformations are exactly the idempotents of the $\mathscr D$-classes $D$ and $C$, respectively. It is routine to calculate that
\begin{align*}
&\pi(B_i,C)=\nabla(\zeta_{i-1},\zeta_i),\\
&\rho(B_i,C)=\Delta(\zeta_{i-1},\zeta_i),\\
&\pi(B_i,D)=\nabla(\eta_{i-1},\eta_i)\cup \nabla(\eta_{n+i},\eta_{n+1+i})\cup \nabla(\eta_{2n+1+i},\eta_{2n+2+i}),\\
&\rho(B_i,D)=\nabla(\eta_{i-1},\eta_{n+1+i},\eta_{2n+1+i})\cup\nabla(\eta_i,\eta_{n+i},\eta_{2n+2+i}),\\
&\pi(C,D)=\nabla(\eta_0,\eta_1,\dots,\eta_n)\cup \nabla(\eta_{2n+2},\eta_{2n+3},\dots,\eta_{3n+2}),\\
&\rho(C,D)=\bigcup_{j=0}^n\nabla(\eta_j,\eta_{2n+2+j}),\\
&\pi(E,D)=\nabla(\eta_n,\eta_{2n+1})\cup \nabla(\eta_{n+1},\eta_{2n+2}),\\
&\rho(E,D)=\nabla(\eta_n,\eta_{n+1})\cup \nabla(\eta_{2n+1},\eta_{2n+2}),\\
&\rho(A,A)=\Delta(E(A))
\end{align*}
for all $i=1,2,\dots,k-1,k+1,\dots,n$ and $A\in T_n(k)/\mathscr D$.

Now we can start looking for a filter $\mathcal K\subseteq [Y)_{T_n(k)}$ that excludes $X$ and separates $\mu$ from $\nu$. If $Y\ne D$, then the empty filter $\mathcal K$ does the job. Indeed, $X \notin \mathcal K$ and the relation
\[
\bigcup\{\pi(X_1,Y) \mid X_1 \in\mathcal K\}
\]
is empty. The relation $\tau(\mathcal K ,Y)$ must then reduce to the transitive closure of
\[
\bigcup\left\{\rho(X_2,Y) \mid X_2 \in [Y)_{T_n(k)}\right\}.
\]
In view of the above, the latter set is the equality relation on $E(Y)$ in any case. Hence $\tau(\mathcal K ,Y)=\Delta(E(Y))$. Thus, $\mathcal K$ separates $\mu$ from $\nu$.

So, it remains to consider the case when $Y=D$. Then $\mu=\eta_s$ and $\nu=\eta_t$ for some $0\le s,t\le 3n+2$. Four cases are possible.

\smallskip

\textbf{Case 1:} $X=B_m$ for some $m\ne k$. By symmetry, we may assume without any loss that $m<k$.
Two subcases are possible.

\smallskip

\textbf{Case 1.1:} $\gamma=\chi_m\chi_m^{-1}$ and so $s\in\{m-1,n+1+m,2n+1+m\}$ but $t\notin\{m-1,n+1+m,2n+1+m\}$.
Four subcases are possible.

\smallskip

\textbf{Case 1.1.1:} $t\in\{n+m+2,\dots, n+k\}$. This is only possible when $m+1<k$. In this case, the filter
\[
\mathcal K_1:=\{B_i\mid 1\le i\le n,\, i\ne m,m+1,k\}
\]
does the job. Indeed, $X \notin \mathcal K_1$ and
\[
\begin{aligned}
&\bigcup\{\pi(X_1,Y) \mid X_1 \in\mathcal K_1\}=\bigcup_{j=1,\, j\ne m,m+1,k}^n\pi(B_j,D)=\\
&\bigcup_{j=1,\, j\ne m,m+1,k}^n\biggl(\nabla(\eta_{j-1},\eta_{j})\cup\nabla(\eta_{n+j},\eta_{n+1+j})\cup\nabla(\eta_{2n+1+j},\eta_{2n+2+j})\biggr),\\
&\bigcup\{\rho(X_2,Y) \mid X_2 \in[Y)_{T_n(k)}\setminus\mathcal K_1\}=\\
&\rho(E,D)\cup\rho(D,D)\cup\rho(C,D)\cup \rho(B_m,D)\cup \rho(B_{m+1},D)=\\
&\nabla(\eta_n,\eta_{n+1})\cup \nabla(\eta_{2n+1},\eta_{2n+2})\cup\Delta(E(D))\cup\biggl(\bigcup_{j=0}^n\nabla(\eta_{j},\eta_{2n+2+j})\biggr)\cup \\
&\nabla(\eta_{m-1},\eta_{n+1+m},\eta_{2n+1+m})\cup\nabla(\eta_m,\eta_{n+m},\eta_{2n+2+m})\cup\\
&\nabla(\eta_m,\eta_{n+2+m},\eta_{2n+2+m})\cup\nabla(\eta_{m+1},\eta_{n+1+m},\eta_{2n+2+m}).
\end{aligned}
\]
Then the transitive closure of the relation $\bigcup\{\pi(X_1,Y) \mid X_1 \in\mathcal K_1\}$ is
\[
\begin{aligned}
&\nabla(\eta_0,\eta_1,\dots,\eta_{m-1})\cup\nabla(\eta_{m+1},\eta_{m+2},\dots,\eta_{k-1})\cup\nabla(\eta_k,\eta_{k+1},\dots,\eta_n)\cup\\
&\nabla(\eta_{n+1},\eta_{n+2},\dots,\eta_{n+m})\cup\nabla(\eta_{n+m+2},\eta_{n+m+3},\dots,\eta_{n+k})\cup\\
&\nabla(\eta_{n+k+1},\eta_{n+k+2},\dots,\eta_{2n+1})\cup\nabla(\eta_{2n+2},\eta_{2n+3},\dots,\eta_{2n+m+1})\cup\\
&\nabla(\eta_{2n+m+3},\eta_{2n+m+4},\dots,\eta_{2n+k+1})\cup\nabla(\eta_{2n+k+2},\eta_{2n+k+3},\dots,\eta_{3n+2}),
\end{aligned}
\]
while the transitive closure of the relation $\bigcup\{\rho(X_2,Y) \mid X_2 \in[Y)_{T_n(k)}\setminus\mathcal K_1\}$ is
\[
\begin{aligned}
&\nabla(\eta_n,\eta_{n+1},\eta_{3n+2})\cup \nabla(\eta_0,\eta_{2n+1},\eta_{2n+2})\cup\\
&\Delta(E(D))\cup\biggl(\bigcup_{j=1, j\ne m-1,m}^{n-1}\nabla(\eta_{j},\eta_{2n+2+j})\biggr)\cup\\
&\nabla(\eta_{m-1},\eta_{m+1},\eta_{n+m+1},\eta_{2n+m+1},\eta_{2n+m+3})\cup\\
&\nabla(\eta_m,\eta_{n+m},\eta_{n+m+2},\eta_{2n+m+2}).
\end{aligned}
\]
It follows that the relation $\tau(\mathcal K_1,Y)$ coincides with
\[
\begin{aligned}
\nabla(&\eta_0,\eta_1,\dots,\eta_{m-1},\eta_{m+1},\eta_{m+2},\dots,\eta_{k-1},\eta_{n+m+1},\eta_{n+k+1},\eta_{n+k+1},\dots,\eta_{2n+m+1},\\
&\eta_{2n+m+3},\eta_{2n+m+4},\dots,\eta_{2n+k+1})\cup\\
\nabla(&\eta_m,\eta_k,\eta_{k+1},\dots, \eta_{n+m},\eta_{n+m+2},\eta_{n+m+3},\dots,\eta_{n+k},\\&\eta_{2n+m+2},\eta_{2n+k+2},\eta_{2n+k+3},\dots,\eta_{3n+2}).
\end{aligned}
\]
The two classes of the relation $\tau(\mathcal K_1 ,D)$ are shown in Fig.~\ref{fig:tau(K1,D)}; the elements lying in the first [respectively, second] $\tau(\mathcal K_1 ,D)$-class are connected by red [respectively, blue] lines.
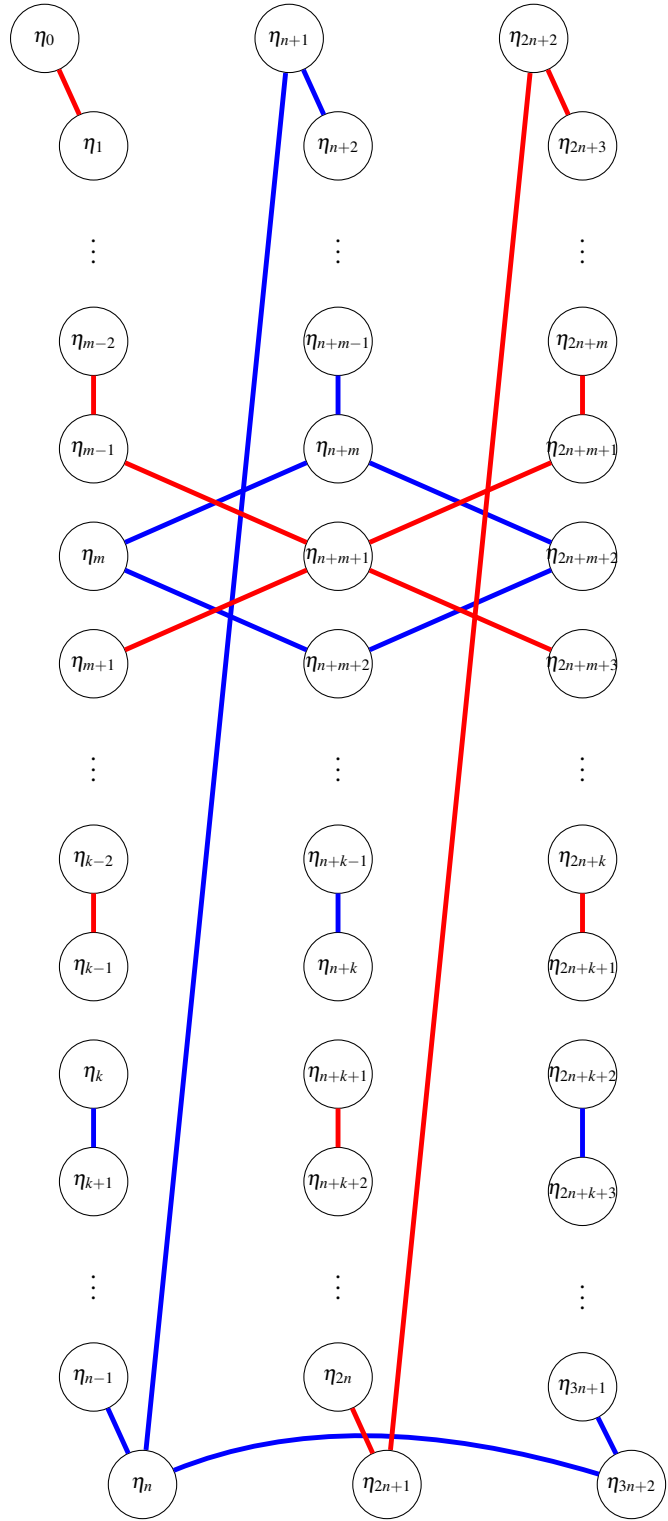
\begin{figure}
\begin{center}
\unitlength 1.3mm
\begin{picture}(60,157)(0,-2.5)
\gasset{Nw=7,Nh=7,Nmr=5}
\node(un)(10,2){\footnotesize $\eta_n$}
\node(un-1)(5,13){\footnotesize $\eta_{n-1}$}
\node[Nframe=n](uk+2)(5,23){$\vdots$}
\node(uk+1)(5,33){\footnotesize $\eta_{k+1}$}
\node(uk)(5,44){\footnotesize $\eta_{k}$}
\node(uk-1)(5,55){\footnotesize $\eta_{k-1}$}
\node(uk-2)(5,66){\footnotesize $\eta_{k-2}$}
\node[Nframe=n](um+2)(5,76){$\vdots$}
\node(um+1)(5,86){\footnotesize $\eta_{m+1}$}
\node(um)(5,97){\footnotesize $\eta_{m}$}
\node(um-1)(5,108){\footnotesize $\eta_{m-1}$}
\node(um-2)(5,119){\footnotesize $\eta_{m-2}$}
\node[Nframe=n](u2)(5,129){$\vdots$}
\node(u1)(5,139){\footnotesize $\eta_1$}
\node(u0)(0,150){\footnotesize $\eta_0$}

\node(u2n+1)(35,2){\footnotesize $\eta_{2n+1}$}
\node(u2n)(30,13){\footnotesize $\eta_{2n}$}
\node[Nframe=n](un+3+k)(30,23){$\vdots$}
\node(un+k+2)(30,33){\footnotesize $\eta_{n+k+2}$}
\node(un+k+1)(30,44){\footnotesize $\eta_{n+k+1}$}
\node(un+k)(30,55){\footnotesize $\eta_{n+k}$}
\node(un+k-1)(30,66){\footnotesize $\eta_{n+k-1}$}
\node[Nframe=n](un+m+3)(30,76){$\vdots$}
\node(un+m+2)(30,86){\footnotesize $\eta_{n+m+2}$}
\node(un+m+1)(30,97){\footnotesize $\eta_{n+m+1}$}
\node(un+m)(30,108){\footnotesize $\eta_{n+m}$}
\node(un+m-1)(30,119){\footnotesize $\eta_{n+m-1}$}
\node[Nframe=n](un+3)(30,129){$\vdots$}
\node(un+2)(30,139){\footnotesize $\eta_{n+2}$}
\node(un+1)(25,150){\footnotesize $\eta_{n+1}$}

\node(u3n+2)(60,2){\footnotesize $\eta_{3n+2}$}
\node(u3n+1)(55,12){\footnotesize $\eta_{3n+1}$}
\node[Nframe=n](u3n)(55,22){$\vdots$}
\node(u2n+k+3)(55,32){\footnotesize $\eta_{2n+k+3}$}
\node(u2n+k+2)(55,44){\footnotesize $\eta_{2n+k+2}$}
\node(u2n+k+1)(55,55){\footnotesize $\eta_{2n+k+1}$}
\node(u2n+k)(55,66){\footnotesize $\eta_{2n+k}$}
\node[Nframe=n](u2n+4+m)(55,76){$\vdots$}
\node(u2n+m+3)(55,86){\footnotesize $\eta_{2n+m+3}$}
\node(u2n+m+2)(55,97){\footnotesize $\eta_{2n+m+2}$}
\node(u2n+m+1)(55,108){\footnotesize $\eta_{2n+m+1}$}
\node(u2n+m)(55,119){\footnotesize $\eta_{2n+m}$}
\node[Nframe=n](u2n+4)(55,129){$\vdots$}
\node(u2n+3)(55,139){\footnotesize $\eta_{2n+3}$}
\node(u2n+2)(50,150){\footnotesize $\eta_{2n+2}$}

\gasset{AHnb=0,linewidth=.5,linecolor=blue}
\drawqbedge(un,30,12,u3n+2){}
\drawedge(um,un+m){}
\drawedge(um,un+m+2){}
\drawedge(uk,uk+1){}
\drawedge(un-1,un){}
\drawedge(un+1,un+2){}
\drawedge(un+m-1,un+m){}
\drawedge(un+m,u2n+m+2){}
\drawedge(un+m+2,u2n+m+2){}
\drawedge(un+k-1,un+k){}
\drawedge(u2n+k+2,u2n+k+3){}
\drawedge(u3n+1,u3n+2){}
\drawedge(un,un+1){}

\gasset{linewidth=.5,linecolor=red}
\drawedge(u0,u1){}
\drawedge(um-2,um-1){}
\drawedge(um-1,un+m+1){}
\drawedge(um+1,un+m+1){}
\drawedge(uk-2,uk-1){}
\drawedge(un+m+1,u2n+m+1){}
\drawedge(un+m+1,u2n+m+3){}
\drawedge(un+k+1,un+k+2){}
\drawedge(u2n,u2n+1){}
\drawedge(u2n+2,u2n+3){}
\drawedge(u2n+m,u2n+m+1){}
\drawedge(u2n+k,u2n+k+1){}
\drawedge(u2n+2,u2n+1){}
\end{picture}
\caption{The two classes of the relation $\tau(\mathcal K_1,D)$}
\label{fig:tau(K1,D)}
\end{center}
\end{figure}
Since $t\in\{n+m+2,\dots, n+k\}$, we see that $\mathcal K_1$ separates $\mu=\eta_s$ from $\nu=\eta_t$.

\smallskip

\textbf{Case 1.1.2:} $t\in\{m,\dots,k-1,n+1,\dots,n+m,2n+m+2,\dots,2n+k+1\}$. In this case, the filter
\[
\mathcal K_2:=\{B_i\mid 1\le i\le n,\, i\ne m,k\}
\]
does the job. Indeed, $X \notin \mathcal K_2$ and
\[
\begin{aligned}
&\bigcup\{\pi(X_1,Y) \mid X_1 \in\mathcal K_2\}=\bigcup_{j=1,\, j\ne m,k}^n\pi(B_j,D)=\\
&\bigcup_{j=1,\, j\ne m,k}^n\biggl(\nabla(\eta_{j-1},\eta_{j})\cup\nabla(\eta_{n+j},\eta_{n+1+j})\cup\nabla(\eta_{2n+1+j},\eta_{2n+2+j})\biggr),\\
&\bigcup\{\rho(X_2,Y) \mid X_2 \in[Y)_{T_n(k)}\setminus\mathcal K_2\}=\\
&\rho(E,D)\cup\rho(D,D)\cup\rho(C,D)\cup \rho(B_m,D)=\\
&\nabla(\eta_n,\eta_{n+1})\cup \nabla(\eta_{2n+1},\eta_{2n+2})\cup\Delta(E(D))\cup\biggl(\bigcup_{j=0}^n\nabla(\eta_{j},\eta_{2n+2+j})\biggr)\cup \\
&\nabla(\eta_{m-1},\eta_{n+1+m},\eta_{2n+1+m})\cup\nabla(\eta_m,\eta_{n+m},\eta_{2n+2+m}).
\end{aligned}
\]
Then the transitive closure of the relation $\bigcup\{\pi(X_1,Y) \mid X_1 \in\mathcal K_2\}$ is
\[
\begin{aligned}
&\nabla(\eta_0,\eta_1,\dots,\eta_{m-1})\cup\nabla(\eta_{m},\eta_{m+1},\dots,\eta_{k-1})\cup\nabla(\eta_k,\eta_{k+1},\dots,\eta_n)\cup\\
&\nabla(\eta_{n+1},\eta_{n+2},\dots,\eta_{n+m})\cup\nabla(\eta_{n+m+1},\eta_{n+m+2},\dots,\eta_{n+k})\cup\\
&\nabla(\eta_{n+k+1},\eta_{n+k+2},\dots,\eta_{2n+1})\cup\nabla(\eta_{2n+2},\eta_{2n+3},\dots,\eta_{2n+m+1})\cup\\
&\nabla(\eta_{2n+m+2},\eta_{2n+m+3},\dots,\eta_{2n+k+1})\cup\nabla(\eta_{2n+k+2},\eta_{2n+k+3},\dots,\eta_{3n+2}),
\end{aligned}
\]
while the transitive closure of the relation $\bigcup\{\rho(X_2,Y) \mid X_2 \in[Y)_{T_n(k)}\setminus\mathcal K_2\}$ is
\[
\begin{aligned}
&\nabla(\eta_n,\eta_{n+1},\eta_{3n+2})\cup \nabla(\eta_0,\eta_{2n+1},\eta_{2n+2})\cup\\
&\Delta(E(D))\cup\biggl(\bigcup_{j=1, j\ne m-1,m}^{n-1}\nabla(\eta_{j},\eta_{2n+2+j})\biggr)\cup\\
&\nabla(\eta_{m-1},\eta_{n+1+m},\eta_{2n+1+m})\cup\nabla(\eta_m,\eta_{n+m},\eta_{2n+2+m}).
\end{aligned}
\]
It follows that the relation $\tau(\mathcal K_2 ,Y)$ coincides with
\[
\begin{aligned}
&\nabla(\eta_0,\eta_1,\dots,\eta_{m-1},\eta_{n+m+1},\eta_{n+m+2},\dots,\eta_{2n+m+1})\cup\\
&\nabla(\eta_m,\eta_{m+1},\dots, \eta_{n+m},\eta_{2n+m+2},\dots,\eta_{3n+2}).
\end{aligned}
\]
Since $t\in\{m,\dots,k-1,n+1,\dots,n+m,2n+m+2,\dots,2n+k+1\}$, we see that $\mathcal K_2$ separates  $\mu=\eta_s$ from $\nu=\eta_t$.

\smallskip

\textbf{Case 1.1.3:} $t\in\{0,1,\dots, m-2, 2n+2, 2n+3,\dots, 2n+m\}$. This is only possible when $m>1$. In this case, the filter
\[
\mathcal K_3:=\{B_i\mid 1\le i\le n,\, i\ne m-1,m,k\}
\]
does the job. Indeed, $X \notin \mathcal K_3$ and
\[
\begin{aligned}
&\bigcup\{\pi(X_1,Y) \mid X_1 \in\mathcal K_3\}=\bigcup_{j=1,\, j\ne m,m+1,k}^n\pi(B_j,D)=\\
&\bigcup_{j=1,\, j\ne m-1,m,k}^n\biggl(\nabla(\eta_{j-1},\eta_{j})\cup\nabla(\eta_{n+j},\eta_{n+1+j})\cup\nabla(\eta_{2n+1+j},\eta_{2n+2+j})\biggr),\\
&\bigcup\{\rho(X_2,Y) \mid X_2 \in[Y)_{T_n(k)}\setminus\mathcal K_3\}=\\
&\rho(E,D)\cup\rho(D,D)\cup\rho(C,D)\cup \rho(B_{m-1},D)\cup \rho(B_m,D)=\\
&\nabla(\eta_n,\eta_{n+1})\cup \nabla(\eta_{2n+1},\eta_{2n+2})\cup\Delta(E(D))\cup\biggl(\bigcup_{j=0}^n\nabla(\eta_{j},\eta_{2n+2+j})\biggr)\cup \\
&\nabla(\eta_{m-2},\eta_{n+m},\eta_{2n+m})\cup\nabla(\eta_{m-1},\eta_{n-1+m},\eta_{2n+1+m})\cup\\
&\nabla(\eta_{m-1},\eta_{n+1+m},\eta_{2n+1+m})\cup\nabla(\eta_m,\eta_{n+m},\eta_{2n+2+m}).
\end{aligned}
\]
Then the transitive closure of the relation $\bigcup\{\pi(X_1,Y) \mid X_1 \in\mathcal K_3\}$ is
\[
\begin{aligned}
&\nabla(\eta_0,\eta_1,\dots,\eta_{m-2})\cup\nabla(\eta_{m},\eta_{m+1},\dots,\eta_{k-1})\cup\nabla(\eta_k,\eta_{k+1},\dots,\eta_n)\cup\\
&\nabla(\eta_{n+1},\eta_{n+2},\dots,\eta_{n+m-1})\cup\nabla(\eta_{n+m+1},\eta_{n+m+2},\dots,\eta_{n+k})\cup\\
&\nabla(\eta_{n+k+1},\eta_{n+k+2},\dots,\eta_{2n+1})\cup\nabla(\eta_{2n+2},\eta_{2n+3},\dots,\eta_{2n+m})\cup\\
&\nabla(\eta_{2n+m+2},\eta_{2n+m+3},\dots,\eta_{2n+k+1})\cup\nabla(\eta_{2n+k+2},\eta_{2n+k+3},\dots,\eta_{3n+2}),
\end{aligned}
\]
while the transitive closure of the relation $\bigcup\{\rho(X_2,Y) \mid X_2 \in[Y)_{T_n(k)}\setminus\mathcal K_3\}$ is
\[
\begin{aligned}
&\nabla(\eta_n,\eta_{n+1},\eta_{3n+2})\cup \nabla(\eta_0,\eta_{2n+1},\eta_{2n+2})\cup\\
&\Delta(E(D))\cup\biggl(\bigcup_{j=1, j\ne m-2,m-1}^{n-1}\nabla(\eta_{j},\eta_{2n+2+j})\biggr)\cup\\
&\nabla(\eta_{m-2},\eta_{m},\eta_{n+m},\eta_{2n+m},\eta_{2n+m+2})\cup\\
&\nabla(\eta_{m-1},\eta_{n+m-1},\eta_{n+m+1},\eta_{2n+m+1}).
\end{aligned}
\]
It follows that the relation $\tau(\mathcal K_3 ,Y)$ coincides with
\[
\begin{aligned}
\nabla(&\eta_0,\eta_1,\dots,\eta_{m-2},\eta_m,\eta_{m+1},\dots,\eta_{k-1},\eta_{n+m},\\
&\eta_{n+k+1},\eta_{n+k+2},\dots,\eta_{2n+m},\eta_{2n+m+2},\eta_{2n+m+3},\dots,\eta_{2n+k+1})\cup\\
\nabla(&\eta_{m-1},\eta_k,\eta_{k+1},\dots,\eta_{n+m-1},\eta_{n+m+1},\eta_{n+m+2},\dots,\eta_{n+k},\\
&\eta_{2n+m+1},\eta_{2n+k+2},\eta_{2n+k+3},\dots,\eta_{3n+2}).
\end{aligned}
\]
Since $t\in\{0,1,\dots, m-2, 2n+2, 2n+3,\dots, 2n+m\}$, we see that $\mathcal K_3$ separates $\mu=\eta_s$ from $\nu=\eta_t$.

\smallskip

\textbf{Case 1.1.4:} $t\in\{k,\dots, n, n+k+1,\dots, 2n+1,2n+k+2,\dots, 3n+2\}$. In this case, the filter $
\mathcal K_4:=\{E\}$ does the job. Indeed, $X \notin \mathcal K_4$ and
\[
\begin{aligned}
&\bigcup\{\pi(X_1,Y) \mid X_1 \in\mathcal K_4\}=\pi(E,D)=\nabla(\eta_n,\eta_{2n+1})\cup \nabla(\eta_{n+1},\eta_{2n+2}),\\
&\bigcup\{\rho(X_2,Y) \mid X_2 \in[Y)_{T_n(k)}\setminus\mathcal K_4\}=\rho(D,D)\cup\rho(C,D)\cup\biggl(\bigcup_{i=1, i\ne k}^n\rho(B_i,D)\biggr)=\Delta(E(D))\cup\\
&\biggl(\bigcup_{j=0}^n\nabla(\eta_{j},\eta_{2n+2+j})\biggr)\cup \biggl(\bigcup_{i=1, i\ne k}^n\nabla(\eta_{i-1},\eta_{n+1+i},\eta_{2n+1+i})\cup\nabla(\eta_i,\eta_{n+i},\eta_{2n+2+i})\biggr).
\end{aligned}
\]
It follows that the relation $\tau(\mathcal K_4 ,Y)$ coincides with
\[
\begin{aligned}
\nabla(&\eta_0,\dots,\eta_{k-1},\eta_{n+1},\dots,\eta_{n+k},\eta_{2n+2},\dots,\eta_{2n+k+1})\cup\\
\nabla(&\eta_k,\dots,\eta_n,\eta_{n+k+1},\dots,\eta_{2n+1},\eta_{2n+k+2},\dots,\eta_{3n+2}).
\end{aligned}
\]
Since $t\in\{k,\dots, n, n+k+1,\dots, 2n+1,2n+k+2,\dots, 3n+2\}$, we see that $\mathcal K_4$ separates $\mu=\eta_s$ from $\nu=\eta_t$.

\smallskip

\textbf{Case 1.2:} $\gamma=\chi_m^{-1}\chi_m$ and so $s\in\{m,n-1+m,2n+m\}$ but $t\notin\{m,n-1+m,2n+m\}$. We have:
\begin{itemize}
\item if $t\in\{m+1,m+2,\dots,k-1,2n+m+3,2n+m+4,\dots,2n+k+1\}$ (provided that $m+1<k$), then it follows from the calculations in Case 1.1.1 that the filter $\mathcal K_1$ separates $\mu=\eta_s$ from $\nu=\eta_t$;
\item if $t\in\{0,1,\dots,m-1,n+m+1,n+m+1,\dots,n+k,2n+2,\dots,2+m+1\}$, then it follows from the calculations in Case 1.1.2 that the filter $\mathcal K_2$ separates $\mu=\eta_s$ from $\nu=\eta_t$;
\item if $t\in\{n+1,k+2,\dots,n+m-1\}$ (provided that $m>1$), then it follows from the calculations in Case 1.1.3 that the filter $\mathcal K_3$ separates $\mu=\eta_s$ from $\nu=\eta_t$;
\item if $t\in\{k,\dots, n, n+k+1,\dots, 2n+1,2n+k+2,\dots, 3n+2\}$, then it follows from the calculations in Case 1.1.4 that the filter $\mathcal K_4$ separates $\mu=\eta_s$ from $\nu=\eta_t$.
\end{itemize}
This completes the consideration of Case 1.2.

\smallskip

\textbf{Case 2:} $X=C$. In this case, $|s-t|\ne 2n$ because if $|s-t|=2n$, then no idempotent $\gamma\in X$ can satisfy $\mu \le \gamma$ and $\nu\nleq \gamma$. Clearly, either $s\in\{m-1,n+m+1,2n+m+1\}$ or $s\in\{m,n+m,2n+m+2\}$ for some $m=1,2,\dots,n$. We may assume without any loss that $s\in\{m-1,n+m+1,2n+m+1\}$.

Assume that $t\in\{m-1,n+m+1,2n+m+1\}$. Since $|s-t|\ne 2n$, this is only possible when $s$ belongs to one of the sets $\{m-1,2n+m+1\}$ or $\{n+m+1\}$, while $t$ belongs to the other one. Arguments similar to ones from Case 1 can show that
\begin{itemize}
\item if $m=k$, then the filter $\mathcal K_4$ separates $\mu=\eta_s$ from $\nu=\eta_t$;
\item if $m\ne k$, then the filter
\[
\mathcal K_5:=\{B_i\mid 1\le i\le n,\, i\ne k\}
\]
separates $\mu=\eta_s$ from $\nu=\eta_t$.
\end{itemize}

Let now $t\notin\{m-1,n+m+1,2n+m+1\}$. If $m\ne k$, then we can use the arguments in Case 1.
Suppose now that $m=k$. Arguments similar to ones from Case 1 can show that
\begin{itemize}
\item if $t\in\{k,k+1,\dots,n,n+1,n+2,\dots,n+k,2n+k+2,2n+k+3,\dots,3n+2\}$, then the filter $\mathcal K_5$ separates $\mu=\eta_s$ from $\nu=\eta_t$;
\item if $t\in\{n+k+2,n+k+3,\dots,2n+1\}$ (provided that $k<n$), then the filter
\[
\mathcal K_6:=\{B_i\mid 1\le i\le n,\, i\ne k,k+1\}
\]
separates $\mu=\eta_s$ from $\nu=\eta_t$;
\item if $t\in\{0,1,\dots,k-2,2n+2,2n+3,\dots,2n+k\}$ (provided that $k>1$), then the filter
\[
\mathcal K_7:=\{B_i\mid 1\le i\le n,\, i\ne k-1,k\}
\]
separates $\mu=\eta_s$ from $\nu=\eta_t$.
\end{itemize}

\smallskip

\textbf{Case 3:} $X=D$. In this case, $Y=X=D$ because $Y\ne\{0\}$. If $|s-t|\ne 2n$, then we can use the arguments as in Cases 1 and 2. Assume that $|s-t|=2n$. In this case, the filter $\mathcal K_8:=[C)_{T_n(k)}=\{B_i,C\mid 1\le i\le n,\, i\ne k\}$ does the job. Indeed, $X \notin \mathcal K_8$ and
\[
\begin{aligned}
&\bigcup\left\{\pi(X_1,D) \mid X_1 \in [C)_{T_n(k)}\right\}=\\
&\biggl(\bigcup_{i=1,\ i\ne k}^n\nabla(\eta_{i-1},\eta_i)\cup \nabla(\eta_{n+i},\eta_{n+1+i})\cup \nabla(\eta_{2n+1+i},\eta_{2n+2+i})\biggr)\cup\\
&\nabla(\eta_0,\eta_1,\dots,\eta_n)\cup \nabla(\eta_{2n+2},\eta_{2n+3},\dots,\eta_{3n+2}),\\
&\bigcup\{\rho(X_2,Y) \mid X_2 \in[Y)_{T_n(k)}\setminus\mathcal K_8\}=\\
&\rho(D,D)\cup\rho(E,D)=\Delta(E(D))\cup\nabla(\eta_n,\eta_{n+1})\cup \nabla(\eta_{2n+1},\eta_{2n+2}).
\end{aligned}
\]
Hence $\tau(\mathcal K_8 ,Y)=\nabla(\eta_0,\eta_1,\dots,\eta_{n+k})\cup \nabla(\eta_{n+k+1},\eta_{n+k+2},\dots,\eta_{3n+2})$. Since $|s-t|=2n$, we see that $\mathcal K_8$ separates  $\mu=\eta_s$ from $\nu=\eta_t$.

\smallskip

\textbf{Case 4:} $X=E$. Then $\gamma\in\{\chi\chi^{-1},\chi^{-1}\chi\}$. By symmetry, we may assume without any loss that $\gamma=\chi\chi^{-1}$. Hence $s\in \{n,n+1\}$ and $t\notin \{n,n+1\}$.
Arguments similar to the above can show that
\begin{itemize}
\item if $t\in\{0,1,\dots,k-1,n+k+1,n+k+2,\dots,2n+k+1\}$, then the filter $\mathcal K_5$ separates $\mu=\eta_s$ from $\nu=\eta_t$;
\item if $t=3n+2$, then the filter $\mathcal K_8$ separates $\mu=\eta_s$ from $\nu=\eta_t$;
\item if $t\in\{n+2,n+3,\dots,n+k\}$ (provided that $k>1$), then the filter
\[
\mathcal K_9:=\{B_i\mid 1\le i\le n,\, i\ne 1,k\}
\]
separates $\mu=\eta_s$ from $\nu=\eta_t$;
\item if $t\in\{k,k+1,\dots,n-1,2n+k+2,2n+k+3,\dots,3n+1\}$ (provided that $k<n$), then the filter
\[
\mathcal K_{10}:=\{B_i\mid 1\le i\le n,\, i\ne k,n\}
\]
separates $\mu=\eta_s$ from $\nu=\eta_t$;
\end{itemize}
This completes the consideration of Case 4.
\end{proof}

Since the inverse semigroup $(B_2^1,\cdot,{}^{-1})$ satisfies the identity $x^2= x^3$, applying Proposition~\ref{prop:Sn1} to monogenic inverse subsemigroups of $(S_n,\cdot,{}^{-1})$ yields the following fact:

\begin{corollary}
\label{cor:x2x3}
The identity $x^2= x^3$ holds in $(S_n,\cdot,{}^{-1})$ for all $n\ge 2$.
\end{corollary}

Now we state and prove the second property of the semigroups $S_n$ essential for our arguments in Section~\ref{sec:main}. For any $n\ge2$, let
\begin{align*}
v_n&:= \Bigl(x_1x_2\cdots x_{2n}\cdot x_nx_{n-1}\cdots x_1\cdot x_{n+1}x_{n+2}\cdots x_{2n}\Bigr)\cdot x_1x_2\cdots x_{n},\\
v_n^\prime&:= \Bigl(x_1x_2\cdots x_{2n}\cdot x_nx_{n-1}\cdots x_1\cdot x_{n+1}x_{n+2}\cdots x_{2n}\Bigr)^2\cdot x_1x_2\cdots x_{n}
\end{align*}

\begin{proposition}
\label{prop:Sn2}
Let $n\ge 2$. The semigroup $(S_n,\cdot)$ violates the identity $v_n=v_n^\prime$.
\end{proposition}

\begin{proof}
Define a substitution $\varphi_n\colon\{x_1,x_2,\dots,x_{2n}\}\to S_n$ as follows:
\[
\varphi_n(x_i):=
\begin{cases}
\chi_i&\text{if }1\le i\le n,\\
\chi&\text{if }i=n+1,\\
\chi^{-1}\chi&\text{if }n+2\le i\le 2n.
\end{cases}
\]
Using Fig.~\ref{fig:semigroups-sn}, it is easy to check that the transformation $\varphi_n(v_n)$ maps $0$ to $3n+2$, while the transformation $\varphi_n(v_n^\prime)$ is nowhere defined. Therefore, $\varphi_n(v_n)\ne\varphi_n(v_n^\prime)$ whence the identity $v_n=v_n^\prime$ fails in $(S_n,\cdot)$.
\end{proof}

\section{Main results}
\label{sec:main}

We start with a general result ensuring the absence of a finite \ib\ for \ais{}s that generate a variety containing the \ais\ $\mB_2^1$.

\begin{theorem}
\label{thm:main}
Let $\mathcal{S}=(S,+,\cdot)$ be an \ais\ whose multiplicative reduct satisfies the identities $v_n= v_n^\prime$ for all $n\ge2$. If the \ais\ $\mB_2^1$ belongs to the variety $\var\mathcal{S}$, then $\mathcal{S}$ admits no finite identity basis.
\end{theorem}

\begin{proof}
Arguing by contradiction, assume that for some $k$, the \ais\ $\mathcal{S}$ has an identity basis $\Sigma$ such that each identity in $\Sigma$ involves less than $k$ variables. Consider the inverse semigroup $(S_k,\cdot,{}^{-1})$. By Corollary~\ref{cor:x2x3} $(S_k,\cdot,{}^{-1})$ satisfies $x^2= x^3$, and therefore, Lemma~\ref{lem:as-ais} implies that $(S_k,+_{\mathrm{nat}},\cdot)$ is an ai-semi\-ring. We claim that this \ais\ satisfies an arbitrary identity $u=v$ in $\Sigma$.

By Lemma~\ref{lem:as-ais}, $x+_{\mathrm{nat}} y$ expresses as $(xy^{-1})^2x$ in $(S_k,+_{\mathrm{nat}},\cdot)$. Therefore, one can rewrite the identity $u= v$ into an identity $u'= v'$ in which $u'$ and $v'$ are $(\cdot,{}^{-1})$-terms with the same variables as $u$ and $v$. Let $x_1,x_2,\dots,x_\ell$ be all variables that occur in $u'$ or $v'$. Consider an arbitrary substitution $\tau\colon\{x_1,x_2,\dots,x_\ell\}\to S_k$ and let $(T,\cdot,{}^{-1})$ be the inverse subsemigroup of $(S_k,\cdot,{}^{-1})$ generated by the elements $\tau(x_1),\tau(x_2),\dots,\tau(x_\ell)$. Since $\ell<k$, Proposition~\ref{prop:Sn1} implies that $(T,\cdot,{}^{-1})$ belongs to the inverse semigroup variety generated by the 6-element Brandt monoid $(B_2^1,\cdot,{}^{-1})$.

Since by the condition of the theorem, the \ais\ $\mB_2^1$ lies in $\var\mathcal{S}$, the identity $u=v$ holds in $\mB_2^1$. This implies that the rewritten identity $u'= v'$ holds in $(B_2^1,\cdot,{}^{-1})$. (Here we utilize the fact that $(B_2^1,\cdot,{}^{-1})$ satisfies $x^2= x^3$, and therefore, $x+_{\mathrm{nat}} y$ expresses in $(B_2^1,\cdot,{}^{-1})$ as the same $(\cdot,{}^{-1})$-term $(xy^{-1})^2x$.) Hence the identity $u'= v'$ holds also in the inverse semigroup $(T,\cdot,{}^{-1})$, and so $u'$ and $v'$ take the same value under every substitution of elements of $T$ for the variables $x_1,x_2,\dots,x_\ell$. In particular, $\tau( u)=\tau( u')=\tau(v')=\tau(v)$. Since the substitution $\tau$ is arbitrary, this proves our claim that the identity $u= v$ holds in the \ais\ $(S_k,+_{\mathrm{nat}},\cdot)$. Since $u= v$ is an arbitrary identity from the identity basis $\Sigma$ of $\mathcal{S}$, we see that $(S_k,+_{\mathrm{nat}},\cdot)$ satisfies all identities of $\mathcal{S}$. Forgetting the addition, we conclude that the multiplicative reduct $(S_k,\cdot)$  of $(S_k,+_{\mathrm{nat}},\cdot)$ satisfies all identities of the multiplicative reduct $(S,\cdot)$ of $\mathcal{S}$. By the condition of the theorem, $(S,\cdot)$ satisfies the identity $v_k = v_k^\prime$, but by Proposition \ref{prop:Sn2} this identity fails in $(S_k,\cdot)$, a contradiction.
\end{proof}

\begin{theorem}
\label{thm:a21}
The \ais{}s $\mA_2^1$ and $\End(\mC_3)$ admit no finite identity basis.
\end{theorem}

\begin{proof}
We verify that the \ais{}s $\mA_2^1$ and $\End(\mC_3)$ satisfy the premise of Theorem~\ref{thm:main}.

By Corollary~\ref{cor:b21-in-a21}, the \ais\ $\mB_2^1$ lies in the variety $\var\mA_2^1$. Recall that in Section~\ref{subsec:a21}, the \ais\ $\mA_2^1$ was represented as the subsemiring of the \ais\ $\End(\mC_3)$ consisting of all order preserving transformations of the chain $0<1<2$ that fix 2. Hence, $\mB_2^1$ also lies in the variety $\var\End(\mC_3)$.

It remains to verify that each of the semigroups $A_2^1$ and $(\End(\mC_3),\cdot)$ satisfies the identities $v_n= v_n^\prime$, $n=2,3,\dotsc$. For every $m\ge 2$, the $\mathscr D$-classes of the semigroup $(\End(\mC_m),\cdot)$ are $I_1,I_2,\dots,I_m$ where $I_k$ is the set of all transformations of rank $k$ from $\End(\mC_m)$. (This follows, for instance, from \cite[Exercise 2.16 and Proposition~2.4.2]{Howie:1995}.) In particular, the semigroup $(\End(\mC_3),\cdot)$ has three $\mathscr D$-classes shown in Fig.~\ref{fig:O3-D-classes}; the light-gray cells show the subsemigroup isomorphic to $A_2^1$.

\begin{figure}[ht]
\begin{center}
\unitlength 1.35mm
\begin{picture}(89,56)(0,2.5)
\gasset{Nw=13,Nh=10,Nmr=0}
\node[Nframe=n,Nw=15](I3)(30,57){$I_3:$}
\node[fillgray=0.9](I3-1)(40,55){\footnotesize $\begin{pmatrix}0&1&2\\0&1&2\end{pmatrix}$}
\node[Nframe=n,Nw=0,Nh=0](I3-b)(40,50){}

\node[Nframe=n,Nw=0,Nh=0](I2-t)(40,40){}
\node[Nframe=n,Nw=15](I2)(17,37){$I_2:$}
\node(I2-1-1)(27,35){\footnotesize $\begin{pmatrix}0&1&2\\0&1&1\end{pmatrix}$}
\node[fillgray=0.9](I2-1-2)(40,35){\footnotesize $\begin{pmatrix}0&1&2\\0&2&2\end{pmatrix}$}
\node[fillgray=0.9](I2-1-3)(53,35){\footnotesize $\begin{pmatrix}0&1&2\\1&2&2\end{pmatrix}$}
\node(I2-2-1)(27,25){\footnotesize $\begin{pmatrix}0&1&2\\0&0&1\end{pmatrix}$}
\node[fillgray=0.9](I2-2-2)(40,25){\footnotesize $\begin{pmatrix}0&1&2\\0&0&2\end{pmatrix}$}
\node[fillgray=0.9](I2-2-3)(53,25){\footnotesize $\begin{pmatrix}0&1&2\\1&1&2\end{pmatrix}$}
\node[Nframe=n,Nw=0,Nh=0](I2-b)(40,20){}

\node[Nframe=n,Nw=0,Nh=0](I1-t)(40,10){}
\node[Nframe=n,Nw=15](I3)(17,7){$I_1:$}
\node(I3-1)(27,5){\footnotesize $\begin{pmatrix}0&1&2\\0&0&0\end{pmatrix}$}
\node(I3-2)(40,5){\footnotesize $\begin{pmatrix}0&1&2\\1&1&1\end{pmatrix}$}
\node[fillgray=0.9](I3-3)(53,5){\footnotesize $\begin{pmatrix}0&1&2\\2&2&2\end{pmatrix}$}

\drawedge[AHnb=0](I3-b,I2-t){}
\drawedge[AHnb=0](I2-b,I1-t){}
\end{picture}
\caption{The partially ordered set $\End(\mC_3)/\mathscr D$;  the light-gray cells show the subsemigroup isomorphic to $A_2^1$}
\label{fig:O3-D-classes}
\end{center}
\end{figure}
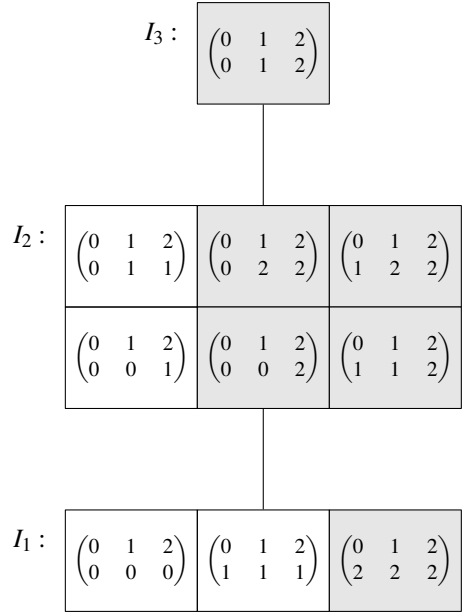

The least $\mathscr D$-class $I_1$ of $(\End(\mC_3),\cdot)$ consists of three constant transformations, which are right zeros, and constitutes an ideal in $(\End(\mC_3),\cdot)$. The Rees quotient  $\End(\mC_3)/I_1$  has a subsemigroup isomorphic to $A_2^1$; on the other hand, it is easy to see that the Rees quotient divides the direct product $A_2^1\times A_2^1$. Hence, $\End(\mC_3)/I_1$ and $A_2^1$ satisfy the same semigroup identities. A characterization of identities of the semigroup $A_2^1$ is known; see \cite{SS06}. We reproduce the characterization, using the terminology suggested in \cite{Chen20}.

As usual, we refer to semigroup terms as \emph{words}. Given a word $w$, its \emph{first} (\emph{last}) \emph{occurrence word} is obtained from $w$ by retaining only the first [respectively, the last] occurrence of each variable from $\alf(w)$. A \emph{jump} is a triple $(x,G,y)$, where $x$ and $y$ are (not necessarily distinct) variables and $G$ is a (possibly empty) set of variables containing neither $x$ nor $y$. The jump $(x,G,y)$ \emph{occurs} in a word $w$ if $w$ can be factorized as $w=u_{\mathrm{left}}xu_{\mathrm{middle}}yu_{\mathrm{right}}$ where $u_{\mathrm{left}},u_{\mathrm{middle}},u_{\mathrm{right}}$ are (possibly empty) words and $G=\alf(u_{\mathrm{middle}})$.

\begin{proposition}[{\mdseries Reformulation of \cite[Theorem 4.6]{SS06}}]
\label{prop:jump}
An identity $w=w'$ holds in the semigroup $A_2^1$ if and only if $w$ and $w'$ have the same first occurrence and the same last occurrence words, and the same jumps occur in $w$ and $w'$.
\end{proposition}

Inspecting the words $v_n$ and $v_n^\prime$, one readily sees that they have the same first occurrence word $x_1x_2\cdots x_{2n}$ and the same last occurrence word $x_{n+1}x_{n+2}\cdots x_{2n}\cdot x_1x_2\cdots x_{n}$. It remains to verify that the same jumps occur in $v_n$ and $v_n^\prime$. Fix an $n\ge 2$ and let $U:=x_1x_2\cdots x_{n}$, $V:=x_{n+1}x_{n+2}\cdots x_{2n}$, and $\overline{U}:=x_nx_{n-1}\cdots x_1$. Then $v_n=UV\overline{U}VU$ while
\begin{equation}\label{eq:vnprime}
v_n^\prime=\lefteqn{\underbrace{\phantom{UV\overline{U}VU}}_{v_n}}UV\overline{U}V\overbrace{UV\overline{U}VU}^{v_n}.
\end{equation}
Since $v_n$ is a prefix (and a suffix) of $v_n^\prime$, every jump occurring in $v_n$ occurs in $v_n^\prime$ as well. To show the converse, take an arbitrary jump $J:=(x_i,G,x_j)$ occurring in $v_n^\prime$ and let $v_n^\prime=u_{\mathrm{left}}x_iu_{\mathrm{middle}}x_ju_{\mathrm{right}}$ be the factorization corresponding to an occurrence of $J$ in $v_n^\prime$. We may assume that the word $u_{\mathrm{left}}x_i$ is a prefix of $UV\overline{U}V$ and the word $x_ju_{\mathrm{right}}$ is a suffix of $V\overline{U}VU$ since otherwise the jump $J$ occurs within either the underbraced or the overbraced parts of \eqref{eq:vnprime}, and each of these parts is nothing but $v_n$. Under this assumption, the overlap $U$ in \eqref{eq:vnprime} is a factor of the word $u_{\mathrm{middle}}$ whence $G=\alf(u_{\mathrm{middle}})\supseteq\alf(U)=\{x_1,x_2,\dots,x_{n}\}$. By the definition of a jump, $x_i,x_j\notin G$ whence $n+1\le i,j\le 2n$ and both $x_i$ and $x_j$ can occur only in $V$. This readily implies that the jump $J$ occurs within the factor $VUV$ of $v_n^\prime$, but then $J$ also occurs within the factor $V\overline{U}V$ since $\alf(U)=\alf(\overline{U})$. The word $V\overline{U}V$ is a factor of $v_n$, and thus, the jump $J$ occurs in $v_n$ also under the assumption we made. Proposition~\ref{prop:jump} now implies that the identity $v_n=v_n^\prime$ holds in $A_2^1$ (and hence, in $\End(\mC_3)/I_1$) for each $n\ge2$.

To prove that the identity $v_n=v_n^\prime$ holds in $(\End(\mC_3),\cdot)$, consider an arbitrary substitution $\tau\colon\{x_1,x_2,\dots,x_{2n}\}\to\End(\mC_3)$. If $\tau(x_i)\in I_1$ for some $i$, then $\tau(v_n)=\tau(v_n^\prime)$ since $\tau(x_i)$ is a right zero of the semigroup $(\End(\mC_3),\cdot)$ and the words $v_n$ and $v_n^\prime$ have equal suffixes following the last occurrence of $x_i$. Suppose that $\tau(x_i)\in I_2\cup I_3$ for all $i=1,2,\dots,2n$. Since the identity $v_n= v_n^\prime$ holds in the Rees quotient $\End(\mC_3)/I_1$, either $\tau(v_n)=\tau(v_n^\prime)$ or both $\tau(v_n)$ and $\tau(v_n^\prime)$ lie in $I_1$. In the latter case,
\[
\tau(v_n^\prime)=\tau(x_1x_2\cdots x_{2n}\cdot x_nx_{n-1}\cdots x_1\cdot x_{n+1}x_{n+2}\cdots x_{2n})\cdot\tau(v_n)=\tau(v_n)
\]
because every element of  the ideal $I_1$, in particular $\tau(v_n)$, is a right zero of the semigroup $(\End(\mC_3),\cdot)$. Since the substitution $\tau$ is arbitrary, the semigroup $(\End(\mC_3),\cdot)$ satisfies the identity $v_n= v_n^\prime$. Now Theorem~\ref{thm:main} applies, yielding that the \ais{}s $\mA_2^1$ and $\End(\mC_3)$ admit no finite identity basis.
\end{proof}

As a consequence, we obtain a complete solution to the FBP for endomorphism semirings of finite chains.

\begin{theorem}
\label{thm:fbpforendo}
The endomorphism semiring of a chain with $m$ elements admits a finite identity basis if and only if $m\le 2$.
\end{theorem}

\begin{proof}
By Proposition~\ref{prop:infbmge4}, the \ais{} $\End(\mC_m)$ is \nfb{} (and even inherently \nfb) for each $m\ge 4$. Theorem~\ref{thm:a21} shows that $\End(\mC_3)$ also is \nfb. The  \ais{} $\End(\mC_1)$ has one element, and thus, it is trivially \fb.  The \ais{} $\End(\mC_2)$ has three elements; its additive semilattice is a chain while its multiplicative reduct is the 2-element right zero semigroup with identity element adjoined. The dual of $\End(\mC_2)$ appears in \cite{GPZ05} where it is denoted by $B$. As \cite[Theorem 2.11]{GPZ05} provides a finite identity basis for $B$, the \ais{} $\End(\mC_2)$ also is \fb{}.
\end{proof}

\begin{remark}
\label{rem:dld}
So far, we have considered the set $\End(\mC_m)$ of all endomorphisms of the chain $0<1<\dots<m-1$ as an \ais{} under composition $\cdot$ and point-wise addition $+$. Yet another natural semilattice operation on the same set can be introduced as follows: for $\alpha,\beta\in\End(\mC_m)$, define $\alpha\diamondsuit\beta$ by $x(\alpha\diamondsuit\beta):=\min\{x\alpha,x\beta\}$ for each $x=0,1,\dots,m-1$. The structure $(\End(\mC_m),\diamondsuit,\cdot)$ is, of course, nothing but the \ais{} of all endomorphisms of the chain $0>1>\dots>m-1$, and so it is isomorphic to $(\End(\mC_m),+,\cdot)$. Thus, Theorem~\ref{thm:fbpforendo} applies to $(\End(\mC_m),\diamondsuit,\cdot)$ too.

One may consider the combined structure $(\End(\mC_m),+,\diamondsuit,\cdot)$ and ask for a solution to the FBP for $(2,2,2)$-structures of this kind. Here the answer is an immediate consequence of a classic fact of universal algebra. The reduct $(\End(\mC_m),+,\diamondsuit)$ is readily seen to be a (distributive) lattice, and therefore, each structure $(\End(\mC_m),+,\diamondsuit,\cdot)$ lies in a congruence-distributive variety. Now Baker's finite basis theorem \cite[Theorem V.4.18]{BuSa81} applies, showing that the structure $(\End(\mC_m),+,\diamondsuit,\cdot)$ is \fb{} for each $m$.
\end{remark}

\paragraph*{\textbf{Acknowledgements.}}
The authors are grateful to the referee for his/her remarks and suggestions which helped us discover and fix an error made in the initial version of the paper.

\end{document}